\documentclass[secthm,seceqn,amsthm,ussrhead,12pt]{amsart}
\usepackage[utf8]{inputenc}
\usepackage[english]{babel}
\usepackage{amssymb,amsmath,amsthm,amsfonts,xcolor,enumerate,hyperref,comment,longtable,cleveref}

\usepackage{times}
\usepackage{cite}
\usepackage{pdflscape}
\usepackage{ulem}
\usepackage[mathcal]{euscript}
\usepackage{tikz}
\usepackage{hyperref}
\usepackage{cancel}
\usepackage{stmaryrd}

\usetikzlibrary{arrows}

\newtheorem*{theoremA}{Theorem A}
\newtheorem*{theoremB}{Theorem B}

\newtheorem{theorem}{Theorem}
\newtheorem{lemma}[theorem]{Lemma}

\newtheorem{remark}[theorem]{Remark}

\newtheorem{definition}[theorem]{Definition}

\newenvironment{Proof}[1][Proof.]{\begin{trivlist}
\item[\hskip \labelsep {\bfseries #1}]}{\flushright
$\Box$\end{trivlist}}

\usepackage{stmaryrd}
\usepackage{xcolor}

\begin{document}
\noindent{\Large
The algebraic and geometric classification of \\
nilpotent left-symmetric  algebras}
\footnote{
The work was supported by
 Russian Science Foundation under grant 19-71-10016.
  The authors thank  the referee for   constructive comments.}

   \

   {\bf  Jobir Adashev$^{a}$,
   Ivan Kaygorodov$^{b,c,d}$,
   Abror Khudoyberdiyev$^{a,e}$ \&
   Aloberdi Sattarov$^{a}$}

\
{\tiny

$^{a}$ Institute of Mathematics Academy of
Sciences of Uzbekistan, Tashkent, Uzbekistan.

$^{b}$ CMCC, Universidade Federal do ABC, Santo Andr\'e, Brazil.

$^{c}$ Moscow Center for Fundamental and Applied Mathematics, Moscow, GSP-1, 119991, Russia.

$^{d}$ Saint Petersburg  University, Russia.

$^{e}$ National University of Uzbekistan, Tashkent, Uzbekistan.

\smallskip

   E-mail addresses:

\smallskip

Jobir Adashev (adashevjq@mail.ru)

Ivan   Kaygorodov (kaygorodov.ivan@gmail.com)

Abror Khudoyberdiyev (khabror@mail.ru)

Aloberdi Sattarov (saloberdi90@mail.ru)

}

\

\noindent{\bf Abstract}:
{\it This paper is devoted to the complete algebraic and geometric classification of complex $4$-dimensional nilpotent left-symmetric  algebras.
 The corresponding geometric variety has dimension $15$ and decomposes into $3$ irreducible components determined by the Zariski closures of two one-parameter families of algebras and a two-parameter family of algebras (see Theorem B). 
 In particular, there are no rigid $4$-dimensional complex nilpotent left  symmetric algebras.}

\

\noindent {\bf Keywords}:
{\it left-symmetric algebras,  Novikov algebras, assosymmetric algebras,
nilpotent algebras, algebraic classification, central extension, geometric classification, degeneration.}

\

\noindent {\bf MSC2010}: 17D25, 17A30, 14D06, 14L30.

\section*{Introduction}

The algebraic classification (up to isomorphism) of algebras of dimension $n$ from a certain variety
defined by a certain family of polinomial identities is a classic problem in the theory of non-associative algebras.
There are many results related to the algebraic classification of small-dimensional algebras in the varieties of
Jordan, Lie, Leibniz, Zinbiel and many other algebras \cite{ kkl20, ckls, kkp20, ck13, degr3, usefi1,  degr1, degr2,     ha16, hac18, kv16, krh14}.
Another interesting direction in the classification of algebras is the geometric classification.
There are many results related to the geometric classification of
Jordan, Lie, Leibniz, Zinbiel and many other algebras
\cite{wolf1, kv19, casas,bb14,    BC99,      GRH,   GRH2,    ikv17, ckls, kkp20, ikv18,      kv16, S90}.
An algebraic classification of complex $3$-dimensional left-symmetric algebras is given in \cite{bai}.
In the present paper, we give the algebraic  and geometric classification of
$4$-dimensional nilpotent left-symmetric  algebras.

Left-symmetric algebras (or under other names like Koszul–Vinberg algebras, quasi-associative algebras, pre-Lie algebras, and so on) are a class of nonassociative algebras coming from the study of several topics in geometry and algebra, such as rooted tree algebras, convex homogeneous cones \cite{vin63}, affine manifolds and affine structures on Lie groups \cite{koz61}, deformation of associative algebras \cite{ger63}, and so on. They are Lie-admissible algebras (in the sense that the commutators define Lie algebra structures) whose left multiplication operators form a Lie algebra.
The class of left-symmetric algebras contains  associative algebras,
Novikov algebras and  assosymmetric algebras as subvarieties.
The variety of left-symmetric algebras is defined by the following identity:
\[
\begin{array}{rcl}
(xy)z-x(yz) &=& (yx)z-y(xz).
\end{array} \]
The variety of Novikov algebras is a subclass of  left-symmetric algebras defined by the following identity:
\[
\begin{array}{rcl}
(xy)z &=& (xz)y.
\end{array} \]

Furthermore, left-symmetric algebras are a kind of natural algebraic systems appearing in many fields in mathematics and mathematical physics. Perhaps this is one of the most attractive and interesting places. As it was pointed out in a paper of Chapoton and Livernet \cite{chap01}, the left-symmetric algebra “deserves more attention than it has been given.” For example, left-symmetric algebras appear as an underlying structure of those Lie algebras that possess a phase space, thus “they form a natural category from the point of view of classical and quantum mechanics” \cite{kuper}; they are the underlying algebraic structures of vertex algebras \cite{bakac03}; there is a correspondence between left-symmetric algebras and complex product structures on Lie algebras \cite{andrada}, which plays an important role in the theory of hypercomplex and hypersymplectic manifolds; left-symmetric algebras have close relations with certain integrable systems \cite{borde}, classical and quantum Yang–Baxter equation \cite{etingof};  Poisson brackets and infinite-dimensional Lie algebras \cite{balnov, geldor}, operads \cite{chap01}, quantum field theory \cite{connes}, and so on (see \cite{burde} and the references therein).

Our method for classifying nilpotent left-symmetric algebras is based on the calculation of central extensions of nilpotent algebras of smaller dimensions from the same variety.
The algebraic study of central extensions of Lie and non-Lie algebras has been an important topic for years \cite{ hac16,ss78,zusmanovich}.
First, Skjelbred and Sund used central extensions of Lie algebras to obtain a classification of nilpotent Lie algebras  \cite{ss78}.
After that, using the method described by Skjelbred and Sund,  all non-Lie central extensions of  all $4$-dimensional Malcev algebras were described \cite{hac16}, and also
all non-associative central extensions of $3$-dimensional Jordan algebras.
Note that the Skjelbred-Sund method of central extensions is an important tool in the classification of nilpotent algebras,
which was used to describe
all  $4$-dimensional nilpotent associative algebras \cite{degr1},
all  $5$-dimensional nilpotent Jordan algebras \cite{ha16},
all  $5$-dimensional nilpotent restricted Lie algebras \cite{usefi1},
all $6$-dimensional nilpotent Lie algebras \cite{degr3,degr2},
all  $6$-dimensional nilpotent Malcev algebras \cite{hac18}
and some others.

\section{The algebraic classification of nilpotent left-symmetric algebras}
\subsection{Method of classification of nilpotent algebras}
Throughout this paper, we use the notations and methods well written in \cite{hac16},
which we have adapted for the left-symmetric case with some modifications.
Further in this section we give some important definitions.

Let $({\bf A}, \cdot)$ be a left-symmetric  algebra over  $\mathbb C$
and $\mathbb V$ a vector space over ${\mathbb C}$. The $\mathbb C$-linear space ${\rm Z^{2}}\left(
\bf A,\mathbb V \right) $ is defined as the set of all  bilinear maps $\theta  \colon {\bf A} \times {\bf A} \longrightarrow {\mathbb V}$ such that
\[ \theta(xy,z)-\theta(x,yz)= \theta(yx,z)-\theta(y,xz). \]

These elements will be called {\it cocycles}. For a
linear map $f$ from $\bf A$ to  $\mathbb V$, if we define $\delta f\colon {\bf A} \times
{\bf A} \longrightarrow {\mathbb V}$ by $\delta f  (x,y ) =f(xy )$, then $\delta f\in {\rm Z^{2}}\left( {\bf A},{\mathbb V} \right) $. We define ${\rm B^{2}}\left({\bf A},{\mathbb V}\right) =\left\{ \theta =\delta f\ : f\in {\rm Hom}\left( {\bf A},{\mathbb V}\right) \right\} $.
We define the {\it second cohomology space} ${\rm H^{2}}\left( {\bf A},{\mathbb V}\right) $ as the quotient space ${\rm Z^{2}}
\left( {\bf A},{\mathbb V}\right) \big/{\rm B^{2}}\left( {\bf A},{\mathbb V}\right) $.

\

Let $\operatorname{Aut}({\bf A}) $ be the automorphism group of  ${\bf A} $ and let $\phi \in \operatorname{Aut}({\bf A})$. For $\theta \in
{\rm Z^{2}}\left( {\bf A},{\mathbb V}\right) $ define  the action of the group $\operatorname{Aut}({\bf A}) $ on ${\rm Z^{2}}\left( {\bf A},{\mathbb V}\right) $ by $\phi \theta (x,y)
=\theta \left( \phi \left( x\right) ,\phi \left( y\right) \right) $.  It is easy to verify that
 ${\rm B^{2}}\left( {\bf A},{\mathbb V}\right) $ is invariant under the action of $\operatorname{Aut}({\bf A}).$
 So, we have an induced action of  $\operatorname{Aut}({\bf A})$  on ${\rm H^{2}}\left( {\bf A},{\mathbb V}\right)$.

\

Let $\bf A$ be a left-symmetric  algebra of dimension $m$ over  $\mathbb C$ and ${\mathbb V}$ be a $\mathbb C$-vector
space of dimension $k$. Given a bilinear map $\theta$, define on the linear space ${\bf A}_{\theta } = {\bf A}\oplus {\mathbb V}$ the
bilinear product `` $\left[ -,-\right] _{{\bf A}_{\theta }}$'' by $\left[ x+x^{\prime },y+y^{\prime }\right] _{{\bf A}_{\theta }}=
 xy +\theta(x,y) $ for all $x,y\in {\bf A},x^{\prime },y^{\prime }\in {\mathbb V}$.
The algebra ${\bf A}_{\theta }$ is called a $k$-{\it dimensional central extension} of ${\bf A}$ by ${\mathbb V}$. One can easily check that ${\bf A_{\theta}}$ is a left-symmetric
algebra if and only if $\theta \in {\rm Z^2}({\bf A}, {\mathbb V})$.

Call the
set $\operatorname{Ann}(\theta)=\left\{ x\in {\bf A}:\theta \left( x, {\bf A} \right)+ \theta \left({\bf A} ,x\right) =0\right\} $
the {\it annihilator} of $\theta $. We recall that the {\it annihilator} of an  algebra ${\bf A}$ is defined as
the ideal $\operatorname{Ann}(  {\bf A} ) =\left\{ x\in {\bf A}:  x{\bf A}+ {\bf A}x =0\right\}$. Observe
 that
$\operatorname{Ann}\left( {\bf A}_{\theta }\right) =(\operatorname{Ann}(\theta) \cap\operatorname{Ann}({\bf A}))
 \oplus {\mathbb V}$.

\

The following result shows that every algebra with a non-zero annihilator is a central extension of a smaller-dimensional algebra.

\begin{lemma}
Let ${\bf A}$ be an $n$-dimensional left-symmetric algebra such that $\dim (\operatorname{Ann}({\bf A}))=m\neq0$. Then there exists, up to isomorphism, a unique $(n-m)$-dimensional left-symmetric  algebra ${\bf A}'$ and a bilinear map $\theta \in {\rm Z^2}({\bf A'}, {\mathbb V})$ with $\operatorname{Ann}({\bf A'})\cap\operatorname{Ann}(\theta)=0$, where $\mathbb V$ is a vector space of dimension m, such that ${\bf A} \cong {{\bf A}'}_{\theta}$ and
 ${\bf A}/\operatorname{Ann}({\bf A})\cong {\bf A}'$.
\end{lemma}

\begin{proof}
Let ${\bf A}'$ be a linear complement of $\operatorname{Ann}({\bf A})$ in ${\bf A}$. Define a linear map $P \colon {\bf A} \longrightarrow {\bf A}'$ by $P(x+v)=x$ for $x\in {\bf A}'$ and $v\in\operatorname{Ann}({\bf A})$, and define a multiplication on ${\bf A}'$ by $[x, y]_{{\bf A}'}=P(x y)$ for $x, y \in {\bf A}'$.
For $x, y \in {\bf A}$, we have
\[P(xy)=P((x-P(x)+P(x))(y- P(y)+P(y)))=P(P(x) P(y))=[P(x), P(y)]_{{\bf A}'}. \]

Since $P$ is a homomorphism $P({\bf A})={\bf A}'$ is a left-symmetric algebra and
 ${\bf A}/\operatorname{Ann}({\bf A})\cong {\bf A}'$, which gives us the uniqueness. Now, define the map $\theta \colon {\bf A}' \times {\bf A}' \longrightarrow\operatorname{Ann}({\bf A})$ by $\theta(x,y)=xy- [x,y]_{{\bf A}'}$.
  Thus, ${\bf A}'_{\theta}$ is ${\bf A}$ and therefore $\theta \in {\rm Z^2}({\bf A'}, {\mathbb V})$ and $\operatorname{Ann}({\bf A'})\cap\operatorname{Ann}(\theta)=0$.
\end{proof}

\begin{definition}
Let ${\bf A}$ be an algebra and $I$ be a subspace of $\operatorname{Ann}({\bf A})$. If ${\bf A}={\bf A}_0 \oplus I$
then $I$ is called an {\it annihilator component} of ${\bf A}$.
A central extension of an algebra $\bf A$ without annihilator component is called a {\it non-split central extension}.
\end{definition}

Our task is to find all central extensions of an algebra $\bf A$ by
a space ${\mathbb V}$.  In order to solve the isomorphism problem we need to study the
action of $\operatorname{Aut}({\bf A})$ on ${\rm H^{2}}\left( {\bf A},{\mathbb V}
\right) $. To do that, let us fix a basis $e_{1},\ldots ,e_{s}$ of ${\mathbb V}$, and $
\theta \in {\rm Z^{2}}\left( {\bf A},{\mathbb V}\right) $. Then $\theta $ can be uniquely
written as $\theta \left( x,y\right) =
\displaystyle \sum_{i=1}^{s} \theta _{i}\left( x,y\right) e_{i}$, where $\theta _{i}\in
{\rm Z^{2}}\left( {\bf A},\mathbb C\right) $. Moreover, $\operatorname{Ann}(\theta)=\operatorname{Ann}(\theta _{1})\cap\operatorname{Ann}(\theta _{2})\cap\ldots \cap\operatorname{Ann}(\theta _{s})$. Furthermore, $\theta \in
{\rm B^{2}}\left( {\bf A},{\mathbb V}\right) $\ if and only if all $\theta _{i}\in {\rm B^{2}}\left( {\bf A},
\mathbb C\right) $.
It is not difficult to prove (see \cite[Lemma 13]{hac16}) that given a left-symmetric algebra ${\bf A}_{\theta}$, if we write as
above $\theta \left( x,y\right) = \displaystyle \sum_{i=1}^{s} \theta_{i}\left( x,y\right) e_{i}\in {\rm Z^{2}}\left( {\bf A},{\mathbb V}\right) $ and
$\operatorname{Ann}(\theta)\cap \operatorname{Ann}\left( {\bf A}\right) =0$, then ${\bf A}_{\theta }$ has   a nonzero
annihilator component if and only if $\left[ \theta _{1}\right] ,\left[
\theta _{2}\right] ,\ldots ,\left[ \theta _{s}\right] $ are linearly
dependent in ${\rm H^{2}}\left( {\bf A},\mathbb C\right) $.

\;

Let ${\mathbb V}$ be a finite-dimensional vector space over $\mathbb C$. The {\it Grassmannian} $G_{k}\left( {\mathbb V}\right) $ is the set of all $k$-dimensional
linear subspaces of $ {\mathbb V}$. Let $G_{s}\left( {\rm H^{2}}\left( {\bf A},\mathbb C\right) \right) $ be the Grassmannian of subspaces of dimension $s$ in
${\rm H^{2}}\left( {\bf A},\mathbb C\right) $. There is a natural action of $\operatorname{Aut}({\bf A})$ on $G_{s}\left( {\rm H^{2}}\left( {\bf A},\mathbb C\right) \right) $.
Let $\phi \in \operatorname{Aut}({\bf A})$. For $W=\left\langle
\left[ \theta _{1}\right] ,\left[ \theta _{2}\right] ,\dots,\left[ \theta _{s}
\right] \right\rangle \in G_{s}\left( {\rm H^{2}}\left( {\bf A},\mathbb C
\right) \right) $ define $\phi W=\left\langle \left[ \phi \theta _{1}\right]
,\left[ \phi \theta _{2}\right] ,\dots,\left[ \phi \theta _{s}\right]
\right\rangle $. We denote the orbit of $W\in G_{s}\left(
{\rm H^{2}}\left( {\bf A},\mathbb C\right) \right) $ under the action of $\operatorname{Aut}({\bf A})$ by $\operatorname{Orb}(W)$. Given
\[
W_{1}=\left\langle \left[ \theta _{1}\right] ,\left[ \theta _{2}\right] ,\dots,
\left[ \theta _{s}\right] \right\rangle ,W_{2}=\left\langle \left[ \vartheta
_{1}\right] ,\left[ \vartheta _{2}\right] ,\dots,\left[ \vartheta _{s}\right]
\right\rangle \in G_{s}\left( {\rm H^{2}}\left( {\bf A},\mathbb C\right)
\right),
\]
we easily have that if $W_{1}=W_{2}$, then $ \bigcap\limits_{i=1}^{s}\operatorname{Ann}(\theta _{i})\cap \operatorname{Ann}\left( {\bf A}\right) = \bigcap\limits_{i=1}^{s}
\operatorname{Ann}(\vartheta _{i})\cap\operatorname{Ann}( {\bf A}) $, and therefore we can introduce
the set
\[
{\bf T}_{s}({\bf A}) =\left\{ W=\left\langle \left[ \theta _{1}\right] ,
\left[ \theta _{2}\right] ,\dots,\left[ \theta _{s}\right] \right\rangle \in
G_{s}\left( {\rm H^{2}}\left( {\bf A},\mathbb C\right) \right) : \bigcap\limits_{i=1}^{s}\operatorname{Ann}(\theta _{i})\cap\operatorname{Ann}({\bf A}) =0\right\},
\]
which is stable under the action of $\operatorname{Aut}({\bf A})$.

\

Now, let ${\mathbb V}$ be an $s$-dimensional linear space and let us denote by
${\bf E}\left( {\bf A},{\mathbb V}\right) $ the set of all {\it non-split $s$-dimensional central extensions} of ${\bf A}$ by
${\mathbb V}$. By the above, we can write
\[
{\bf E}\left( {\bf A},{\mathbb V}\right) =\left\{ {\bf A}_{\theta }:\theta \left( x,y\right) = \sum_{i=1}^{s}\theta _{i}\left( x,y\right) e_{i} \ \ \text{and} \ \ \left\langle \left[ \theta _{1}\right] ,\left[ \theta _{2}\right] ,\dots,
\left[ \theta _{s}\right] \right\rangle \in {\bf T}_{s}({\bf A}) \right\} .
\]
We also have the following result, which can be proved as in \cite[Lemma 17]{hac16}.

\begin{lemma}
 Let ${\bf A}_{\theta },{\bf A}_{\vartheta }\in {\bf E}\left( {\bf A},{\mathbb V}\right) $. Suppose that $\theta \left( x,y\right) =  \displaystyle \sum_{i=1}^{s}
\theta _{i}\left( x,y\right) e_{i}$ and $\vartheta \left( x,y\right) =
\displaystyle \sum_{i=1}^{s} \vartheta _{i}\left( x,y\right) e_{i}$.
Then the left-symmetric algebras ${\bf A}_{\theta }$ and ${\bf A}_{\vartheta } $ are isomorphic
if and only if
$$\operatorname{Orb}\left\langle \left[ \theta _{1}\right] ,
\left[ \theta _{2}\right] ,\dots,\left[ \theta _{s}\right] \right\rangle =
\operatorname{Orb}\left\langle \left[ \vartheta _{1}\right] ,\left[ \vartheta
_{2}\right] ,\dots,\left[ \vartheta _{s}\right] \right\rangle .$$
\end{lemma}

This shows that there exists a one-to-one correspondence between the set of $\operatorname{Aut}({\bf A})$-orbits on ${\bf T}_{s}\left( {\bf A}\right) $ and the set of
isomorphism classes in ${\bf E}\left( {\bf A},{\mathbb V}\right) $. Consequently we have a
procedure that allows us, given a left-symmetric algebra ${\bf A}'$ of
dimension $n-s$, to construct all non-split central extensions of ${\bf A}'$. This procedure is:

\begin{enumerate}
\item For a given left-symmetric algebra ${\bf A}'$ of dimension $n-s $, determine ${\rm H^{2}}( {\bf A}',\mathbb {C}) $, $\operatorname{Ann}({\bf A}')$ and $\operatorname{Aut}({\bf A}')$.

\item Determine the set of $\operatorname{Aut}({\bf A}')$-orbits on ${\bf T}_{s}({\bf A}') $.

\item For each orbit, construct the left-symmetric algebra associated with a
representative of it.
\end{enumerate}

The above described method gives all (Novikov and non-Novikov) left-symmetric algebras. But we are interested in developing this method in such a way that it only gives non-Novikov left-symmetric  algebras, because the classification of all Novikov algebras is given in \cite{kkk18}. Clearly, any central extension of a non-Novikov left-symmetric algebra is non-Novikov. But a Novikov algebra may have extensions which are not Novikov algebras. More precisely, let ${\bf  N}$ be a Novikov algebra and $\theta \in {\rm Z_L^2}({\bf  N}, {\mathbb C}).$ Then ${\bf  N}_{\theta }$ is a Novikov algebra if and only if
\begin{equation*}
 \theta(xy,z)=\theta(xz,y),
 \end{equation*}
for all $x,y,z\in {\bf  N}.$ Define the subspace ${\rm Z_N^2}({\bf  N},{\mathbb C})$ of ${\rm Z_L^2}({\bf  N},{\mathbb C})$ by
\begin{equation*}
{\rm Z_N^2}({\bf  N},{\mathbb C}) =\left\{\begin{array}{c} \theta \in {\rm Z_L^2}({\bf  N},{\mathbb C}) : \theta(xy,z)=\theta(xz,y) \text{ for all } x, y,z\in {\bf  N}\end{array}\right\}.
\end{equation*}
Observe that ${\rm B^2}({\bf  N},{\mathbb C})\subseteq{\rm Z_N^2}({\bf  N},{\mathbb C}).$
Let ${\rm H_N^2}({\bf  N},{\mathbb C}) =%
{\rm Z_N^2}({\bf  N},{\mathbb C}) \big/{\rm B^2}({\bf  N},{\mathbb C}).$ Then ${\rm H_N^2}({\bf  N},{\mathbb C})$ is a subspace of $%
{\rm H_L^2}({\bf  N},{\mathbb C}).$ Define
\begin{eqnarray*}
{\bf R}_{s}({\bf  N})  &=&\left\{ {\bf W}\in {\bf T}_{s}({\bf  N}) :{\bf W}\in G_{s}({\rm H_N^2}({\bf  N},{\mathbb C}) ) \right\}, \\
{\bf U}_{s}({\bf  N})  &=&\left\{ {\bf W}\in {\bf T}_{s}({\bf  N}) :{\bf W}\notin G_{s}({\rm H_N^2}({\bf  N},{\mathbb C}) ) \right\}.
\end{eqnarray*}
Then ${\bf T}_{s}({\bf  N}) ={\bf R}_{s}(
{\bf  N})$ $\mathbin{\mathaccent\cdot\cup}$ ${\bf U}_{s}(
{\bf  N}).$ The sets ${\bf R}_{s}({\bf  N}) $
and ${\bf U}_{s}({\bf  N})$ are stable under the action
of $\operatorname{Aut}({\bf  N}).$ Thus, the left-symmetric algebras
corresponding to the representatives of $\operatorname{Aut}({\bf  N}) $%
-orbits on ${\bf R}_{s}({\bf  N})$ are Novikov  algebras,
while those corresponding to the representatives of $\operatorname{Aut}({\bf  N}%
) $-orbits on ${\bf U}_{s}({\bf  N})$ are not
Novikov algebras. Hence, we may construct all non-split non-Novikov left-symmetric algebras $%
\bf{A}$ of dimension $n$ with $s$-dimensional annihilator
from a given left-symmetric algebra $\bf{A}%
^{\prime }$ of dimension $n-s$ in the following way:

\begin{enumerate}
\item If $\bf{A}^{\prime }$ is non-Novikov, then apply the Procedure.

\item Otherwise, do the following:

\begin{enumerate}
\item Determine ${\bf U}_{s}(\bf{A}^{\prime })$ and $%
\operatorname{Aut}(\bf{A}^{\prime }).$

\item Determine the set of $\operatorname{Aut}(\bf{A}^{\prime })$-orbits on ${\bf U%
}_{s}(\bf{A}^{\prime }).$

\item For each orbit, construct the left-symmetric  algebra corresponding to one of its
representatives.
\end{enumerate}
\end{enumerate}

\subsection{Notations}
Let us introduce the following notations. Let ${\bf A}$ be a nilpotent algebra with
a basis $e_{1},e_{2}, \ldots, e_{n}.$ Then by $\Delta_{ij}$\ we will denote the
bilinear form
$\Delta_{ij}:{\bf A}\times {\bf A}\longrightarrow \mathbb C$
with $\Delta_{ij}(e_{l},e_{m}) = \delta_{il}\delta_{jm}.$
The set $\left\{ \Delta_{ij}:1\leq i, j\leq n\right\}$ is a basis for the linear space of
bilinear forms on ${\bf A},$ so every $\theta \in
{\rm Z^2}({\bf A},\bf \mathbb V )$ can be uniquely written as $
\theta = \displaystyle \sum_{1\leq i,j\leq n} c_{ij}\Delta _{{i}{j}}$, where $
c_{ij}\in \mathbb C$.
Let us fix the following notations:

$$\begin{array}{lll}
\bf{L}^{i}_{j}& \mbox{---}& j\mbox{th }i\mbox{-dimensional left-symmetric (non-Novikov)  algebra.} \\
\bf{L}^{i*}_{j}& \mbox{---}& j\mbox{th }i\mbox{-dimensional left-symmetric (Novikov) algebra.} \\

\end{array}$$

\subsection{The algebraic classification of  $3$-dimensional nilpotent left-symmetric algebras}
There are no nontrivial $1$-dimensional nilpotent left-symmetric algebras.
There is only one nontrivial $2$-dimensional nilpotent left-symmetric algebra
(it is the non-split central extension of the $1$-dimensional algebra with zero product):

$$\begin{array}{ll llll}
{\bf L}^{2*}_{01} &:& e_1 e_1 = e_2.\\
\end{array}$$

The classification of all non-split
3-dimensional nilpotent left-symmetric algebras is known:

\begin{longtable}{ll llllllllllll}
${\bf L}^{3*}_{02}$  &$:$& $e_1 e_1=e_3$  & $e_2 e_2=e_3$  \\
${\bf L}^{3*}_{03}$  &$:$& $e_1 e_2=e_3$  & $e_2 e_1=-e_3$   \\
${\bf L}^{3*}_{04}(\lambda)$  &$:$& $e_1 e_1 = \lambda e_3$   & $e_2 e_1=e_3$  & $e_2 e_2=e_3$  \\
${\bf L}^{3*}_{05}$  &$:$&  $e_1 e_1 = e_2$  & $e_2 e_1=e_3$  \\
${\bf L}^{3*}_{06}(\lambda)$ &$:$& $e_1 e_1 = e_2$ & $e_1 e_2=e_3$ & $e_2 e_1=\lambda e_3.$
\end{longtable}

\subsection{Central extensions of 3-dimensional nilpotent left-symmetric  algebras}

\subsubsection{The description of the second cohomology spaces of $3$-dimensional nilpotent left-symmetric algebras}

\
In the following table we give the description of the second cohomology space of  $3$-dimensional nilpotent left-symmetric algebras.

\begin{longtable}{|l |l |l |  }
                \hline
${\bf A}$ & ${\rm H}^2_{\bf N}({\bf A})$ & ${\rm H}^2_{\bf L}({\bf A})$  \\ \hline

${\bf L}_{01}^{3*}$&
$\Big\langle [\Delta_{12}],[\Delta_{13}], [\Delta_{21}], [\Delta_{31}], [\Delta_{33}] \Big\rangle$ &
${\rm H}^2_{\bf N}({\bf L}_{01}^{3*}) \oplus \Big\langle [\Delta_{23}] \Big\rangle$     \\ \hline

${\bf L}_{02}^{3*}$ &
$\Big\langle [\Delta_{12}],[\Delta_{21}], [\Delta_{22}] \Big\rangle$ &
${\rm H}^2_{\bf N}({\bf L}_{02}^{3*}) \oplus \Big\langle [\Delta_{31}],[\Delta_{32}] \Big\rangle$     \\ \hline

${\bf L}_{03}^{3*}$ &
$\Big\langle  [\Delta_{11}],[\Delta_{21}], [\Delta_{22}]\Big\rangle$ &
${\rm H}^2_{\bf N}({\bf L}_{03}^{3*}) \oplus$ \\ 
&&$\Big\langle [\Delta_{31} -2 \Delta_{13}],[\Delta_{32} -2 \Delta_{23}] \Big\rangle$     \\ \hline

${\bf L}_{04}^{3*}(\lambda)_{\lambda\neq 0}$  &
$\Big\langle [\Delta_{11}], [\Delta_{12}], [\Delta_{21}] \Big\rangle$ &
${\rm H}^2_{\bf N}({\bf L}_{04}^{3*}(\lambda)) \oplus$
\\ & & $\Big\langle [\Delta_{13}-\Delta_{31}-\Delta_{32}], [\Delta_{23}+\lambda\Delta_{31}] \Big\rangle$     \\ \hline

${\bf L}_{04}^{3*}(0)$ &\relax
$\Big\langle

[\Delta_{11}], [\Delta_{12}], [\Delta_{21}],  [\Delta_{13}-\Delta_{31}-\Delta_{32}],  [\Delta_{23}]   

\Big\rangle$ &
${\rm H}^2_{\bf N}({\bf L}_{04}^{3*}(0))$  \\ \hline

${\bf L}_{05}^{3*}$&
$\Big\langle  [\Delta_{12}], [\Delta_{13}- \Delta_{31}] \Big\rangle$ &
${\rm H}^2_{\bf N}({\bf L}_{05}^{3*}) \oplus \Big\langle   [\Delta_{22}+\Delta_{31}], [\Delta_{23}] \Big\rangle$ \\ \hline

${\bf L}_{06}^{3*}(\lambda)$ &
$\Big\langle [\Delta_{21}], [(2-\lambda)\Delta_{13}+\lambda(\Delta_{22}+\Delta_{31})]  \Big\rangle$ &
${\rm H}^2_{\bf N}({\bf L}_{06}^{3*}(\lambda)) \oplus \Big\langle [\Delta_{22}+ \Delta_{13}-\Delta_{31}] \Big\rangle$ \\ \hline
\end{longtable}

\begin{remark} Since
${\rm H}^2_{\bf L}({\bf L}_{04}^{3*}(0)) = {\rm H}^2_{\bf N}({\bf L}_{04}^{3*}(0)),$ then central
 extensions of the algebra ${\bf L}_{04}^{3*}(0)$ give us only Novikov algebras.

\end{remark}

\subsubsection{Central extensions of ${\bf L}^{3*}_{01}$}\label{ext-L^{3*}_{01}}
	Let us use the following notations:
	$$
	\nabla_1 = [\Delta_{12}], \quad \nabla_2 = [\Delta_{13}], \quad \nabla_3 = [\Delta_{21}],\quad \nabla_4 = [\Delta_{31}], \quad \nabla_5 = [\Delta_{33}], \quad \nabla_6 = [\Delta_{23}].
	$$
Take $\theta=\sum\limits_{i=1}^6\alpha_i\nabla_i\in {\rm H}_{\bf L}^2({\bf L}^{3*}_{01}).$
	The automorphism group of ${\bf L}^{3*}_{01}$ consists of invertible matrices of the form
	$$
	\phi=
	\begin{pmatrix}
	x &    0  &  0\\
	y &  x^2  &  u\\
	z &  0  &  t
	\end{pmatrix}.
	$$
	
	Since
	$$
	\phi^T\begin{pmatrix}
	0 &  \alpha_1 & \alpha_2\\
	\alpha_3  & 0 & \alpha_6\\
	\alpha_4&  0    & \alpha_5
	\end{pmatrix} \phi=
	\begin{pmatrix}
		\alpha^*   &  \alpha_1^* & \alpha_2^*\\
	\alpha_3^*  & 0 & \alpha_6^*\\
	\alpha_4^*&  0    & \alpha_5^*
	\end{pmatrix},
	$$
 we have that the action of ${\rm Aut} ({\bf L}^{3*}_{01})$  on the subspace
$\langle \sum\limits_{i=1}^6\alpha_i\nabla_i  \rangle$
is given by
$\langle \sum\limits_{i=1}^6\alpha_i^{*}\nabla_i\rangle,$
where
\[\begin{array}{rclcrcl}
\alpha^*_1&=&\alpha_1 x^3, & &
\alpha^*_2&=&\alpha _1 xu+\alpha _2 xt+\alpha _5 zt+\alpha_6 yt,\\
\alpha^*_3&=&\alpha _3 x^3+\alpha _6 x^2 z,& &
\alpha_4^*&=&\alpha _3 xu+\alpha _4 xt+\alpha _5 z t+\alpha _6 zu,\\
\alpha_5^*&=&\alpha _5 t^2+\alpha _6 tu,& &
\alpha_6^*&=&\alpha _6 x^2 t .
\end{array}\]
	
Since ${\rm H}^2_{\bf L}({\bf L}_{01}^{3*}) = {\rm H}^2_{\bf N}({\bf L}_{01}^{3*})\oplus \langle\nabla_6\rangle$ and we are interested only in new algebras, we have  $\alpha_6\neq0.$
	Then putting $z=-\frac{\alpha_3x}{\alpha_6},$ $u=-\frac{\alpha_5t}{\alpha_6}$ and  $y=\frac{(\alpha_3\alpha_5+\alpha_1\alpha_5-\alpha_2\alpha_6)x}{\alpha_6^2},$   we have
	
\begin{center}
$\alpha^*_2=\alpha^*_3=\alpha_5^*=0,$ 
$\alpha^*_1=\alpha _1 x^3,$  $\alpha_4^*=\frac{(\alpha_4\alpha_6-\alpha_3\alpha _5)xt}{\alpha _6},$ and $\alpha_6^*=\alpha _6 x^2t$ .\end{center}
	
	Consider the following cases.
\begin{enumerate}
\item  $\alpha_1\neq 0,$ $\alpha_4\alpha _6-\alpha_3\alpha _5 \neq 0,$ then by choosing $x=\frac{\alpha_4\alpha _6-\alpha_3\alpha _5}{\alpha_6^2},$ $t=\frac{\alpha_1(\alpha_4\alpha _6-\alpha_3\alpha _5)}{\alpha_6^3},$ we have the representative $\langle \nabla_1+\nabla_4+\nabla_6\rangle$.

\item $\alpha_1=0,$ $\alpha_4\alpha _6-\alpha_3\alpha _5 \neq 0,$
then by choosing $x=\frac{\alpha_4\alpha _6-\alpha_3\alpha _5}{\alpha_6^2},$ we have the representative
    $\langle \nabla_4+\nabla_6\rangle$.
   
\item  $\alpha_1\neq 0,$ $\alpha_4\alpha _6-\alpha_3\alpha _5 = 0,$ then by choosing $t=\frac{\alpha_1x}{\alpha_6},$ we have the representative $\langle \nabla_1+\nabla_6\rangle$.

\item $\alpha_1= 0,$ $\alpha_4\alpha _6-\alpha_3\alpha _5 = 0,$ then we have the representative   $\langle \nabla_6\rangle$.
\end{enumerate}

Hence, we have the following distints orbits
\begin{center}
$\langle \nabla_1+\nabla_4+\nabla_6\rangle,$
$\langle \nabla_4+\nabla_6\rangle,$ 
$\langle \nabla_1+\nabla_6\rangle$ and 
$\langle \nabla_6\rangle,$
\end{center}
which give the following new algebras:
\begin{longtable}{llllllll}
${\bf L}^4_{01}$ &: &
$e_1 e_1 = e_2$&  $e_1 e_2=e_4$& $e_2 e_3=e_4$& $e_3e_1=e_4$ \\
\hline${\bf L}^4_{02}$ &: &
$e_1 e_1 = e_2$&  $e_2 e_3=e_4$& $e_3e_1=e_4$ \\
\hline${\bf L}^4_{03}$ &: &
$e_1 e_1 = e_2$&  $e_1 e_2=e_4$& $e_2 e_3=e_4$ \\
\hline${\bf L}^4_{04}$ &: &
$e_1 e_1 = e_2$& $e_2 e_3=e_4$
\end{longtable}

\subsubsection{Central extensions of ${\bf L}^{3*}_{02}$}\label{ext-N^{3*}_{02}}
	Let us use the following notations:
	$$
	\nabla_1 = [\Delta_{12}], \quad \nabla_2 = [\Delta_{21}], \quad \nabla_3 =[\Delta_{22}], \quad \nabla_4 =[\Delta_{31}],\quad \nabla_5 = [\Delta_{32}].
	$$

	Take $\theta=\sum\limits_{i=1}^5\alpha_i\nabla_i\in {\rm H}_{\bf L}^2({\bf L}^{3*}_{02}).$
	The automorphism group of ${\bf L}^{3*}_{02}$ consists of invertible matrices of the form
	$$
	 \phi_1=
	\begin{pmatrix}
	x &    -y  &  0\\
	y &   x  & 0 \\
	z &  t  &  x^2+y^2
	\end{pmatrix} \quad \text{or} \quad \phi_2=
	\begin{pmatrix}
	x &    y  &  0\\
	y &  -x  & 0 \\
	z &  t  &  x^2+y^2
	\end{pmatrix}.
	$$

	Since
		$$
	\phi_1^T\begin{pmatrix}
	0      &  \alpha_1    & 0\\
	\alpha_2  &\alpha_3 & 0 \\
	\alpha_4&  \alpha_5    & 0
	\end{pmatrix} \phi_1= \begin{pmatrix}
  \alpha^* &  \alpha^*_1    & 0\\
	\alpha^*_2    & \alpha^*+ \alpha^*_3      & 0\\
	\alpha^*_4  &  \alpha^*_5       & 0
	\end{pmatrix},
	$$
we have that the action of ${\rm Aut} ({\bf L}^{3*}_{02})^{+}$ (it is the subgroup in ${\rm Aut} ({\bf  L}^{3*}_{02})$  formed  by all automorphisms of the first type) on the subspace
$\langle \sum\limits_{i=1}^5\alpha_i\nabla_i  \rangle$
is given by
$\langle \sum\limits_{i=1}^5\alpha_i^{*}\nabla_i\rangle,$
where
\[\begin{array}{rcl}\alpha^*_1&=&\alpha _1 x^2-\alpha _2 y^2+\alpha _3 xy-\alpha _4 yz+\alpha _5 xz,\\
\alpha^*_2&=&- \alpha_1y^2+ \alpha _2x^2+ \alpha_3xy+\alpha_4 xt+\alpha_5yt,\\
\alpha_3^*&=&-2\alpha_1 xy-2\alpha_2xy+\alpha_3(x^2-y^2)-\alpha_4(xz+yt)-\alpha_5(yz- xt),\\
\alpha_4^*&=&\left (\alpha_4 x +\alpha_5y \right)\left(x^2+y^2\right),\\
\alpha_5^*&=& \left(-\alpha_4y +\alpha_5x\right)\left(x^2+y^2\right).
\end{array}\]

Since ${\rm H^2_{\bf L}({\bf L}_{02}^{3*})} = {\rm H^2_{\bf N}({\bf L}_{02}^{3*})}\oplus \langle\nabla_4, \nabla_5\rangle$ and we are interested only in new algebras, we have  $(\alpha_4, \alpha_5)\neq(0,0).$
Moreover, without loss of
generality, one can assume $\alpha_4\neq0$.
Then we have the following cases.
\begin{enumerate}
\item $\alpha _4^2+\alpha _5^2 \neq 0,$ then by choosing $y=\frac{x\alpha_5}{\alpha_4}, \  t=\frac{(\alpha_1\alpha_5^2-\alpha_2\alpha_4^2-\alpha_3\alpha_4\alpha_5)x}{\alpha_4(\alpha_4^2+\alpha_5^2)}, \  z=\frac{(\alpha_3(\alpha_4^2-\alpha_5^2)-2\alpha_4\alpha_5(\alpha_1+\alpha_2))x}{\alpha_4(\alpha_4^2+\alpha_5^2)}$, we have
	
\begin{center}
$\alpha^{*}_2=\alpha^{*}_3=\alpha^{*}_5=0,$  $\alpha^{*}_1=\frac{(\alpha_1\alpha_4^2 -\alpha_2\alpha_5^2+\alpha_3\alpha_4\alpha_5)x^2}{\alpha_4^2}$ and  $\alpha_4^{*}=\frac{x^3(\alpha _4^2+\alpha _5^2)^2}{\alpha_4^3}.$
\end{center}

\begin{enumerate}
    \item if $\alpha_1\alpha_4^2 -\alpha_2\alpha_5^2+\alpha_3\alpha_4\alpha_5=0,$ then we have the representative $\langle \nabla_4\rangle$;

\item if $\alpha_1\alpha_4^2 -\alpha_2\alpha_5^2+\alpha_3\alpha_4\alpha_5\neq 0,$ then by choosing $x=\frac{\alpha_4(\alpha_1\alpha_4^2 -\alpha_2\alpha_5^2+\alpha_3\alpha_4\alpha_5)}{(\alpha_4^2+\alpha_5^2)^2},$ we have the representative    $\langle \nabla_1+\nabla_4\rangle$;
\end{enumerate}

\item $ \alpha _4^2+\alpha _5^2 = 0,$ i.e., $\alpha_5=\pm i\alpha_4,$ then by choosing \begin{center}
    $t=\frac{\alpha_1y^2-\alpha_2x^2-\alpha_3xy}{\alpha_4(x\pm iy)}, \ z=\frac{-2\alpha_1xy-2\alpha_2xy+\alpha_3(x^2-y^2)+\alpha_4(-y \pm i x)}{\alpha_4(x\pm iy)}, $
    \end{center}
    we have
    \begin{center}
    $\alpha^*_2= \alpha^*_3=0,$   $\alpha^*_1=-\frac{(\alpha_1+\alpha_2\pm i\alpha_3)(x^2+y^2)^2}{(x\pm iy)^2},$\\ $\alpha^*_4=\left(x^2+y^2\right) \left(x\pm iy\right)\alpha_4,$ $ \alpha^*_5=\pm i \left(x^2+y^2\right) \left(x\pm iy\right)\alpha_4.$
    \end{center}

\begin{enumerate}
\item $\alpha_1+\alpha_2\pm i\alpha_3=0$, then we have representative $\langle \nabla_4\pm i\nabla_5\rangle$.

\item $\alpha_1+\alpha_2\pm i\alpha_3\neq0$, then by choosing $x=-\frac{\alpha_1+\alpha_2\pm i\alpha_3}{\alpha_4}$ and $y=0,$ we have representative 
\begin{center}$\langle \nabla_1+\nabla_4\pm i\nabla_5\rangle$.
\end{center}
\end{enumerate}

Since the automorphism $\phi=diag(1,-1,1)$ acts as
\begin{longtable}{lll}
$\phi(\nabla_4+ i \nabla_5)=\langle \nabla_4- i \nabla_5\rangle$ & and & $\phi(\nabla_1+ \nabla_4+ i \nabla_5)=\langle \nabla_1+ \nabla_4- i \nabla_5\rangle$,
\end{longtable}
we have two representatives of distinct orbits  $\langle \nabla_4+ i\nabla_5\rangle$ and $\langle \nabla_1+ \nabla_4+ i \nabla_5\rangle.$

\end{enumerate}

Hence, we have the following distints orbits
\begin{center}
$\langle \nabla_4 \rangle,$
$\langle \nabla_1+\nabla_4\rangle,$ 
$\langle \nabla_4+i\nabla_5\rangle$ and
$\langle \nabla_1+ \nabla_4+ i \nabla_5\rangle$,
\end{center}
which give the following new algebras:

\begin{longtable}{llllllll}
${\bf L}^4_{05}$ &: &
$e_1 e_1 = e_3$&  $e_2 e_2=e_3$&  $e_3e_1=e_4$ \\
\hline ${\bf L}^4_{06}$ &: &
$e_1 e_1 = e_3$&  $e_1 e_2=e_4$&$e_2 e_2=e_3$& $e_3e_1=e_4$ \\
\hline${\bf L}^4_{07}$ &: &
$e_1 e_1 = e_3$&  $e_2 e_2=e_3$& $e_3 e_1=e_4$&$e_3 e_2= ie_4$ \\
\hline ${\bf L}^4_{08}$ &: &
$e_1 e_1 = e_3$& $e_1 e_2=e_4$& $e_2 e_2=e_3$& $e_3 e_1=e_4$&$e_3 e_2= ie_4$
\end{longtable}

\subsubsection{Central extensions of ${\bf L}^{3*}_{03}$}\label{ext-N^{3*}_{03}}
	Let us use the following notations:
	$$
	\nabla_1 =[\Delta_{11}], \quad  \nabla_2 =[\Delta_{21}], \quad \nabla_3 =[\Delta_{22}],\quad \nabla_4 =[\Delta_{31}-2\Delta_{13}], \quad \nabla_5 =[\Delta_{32}-2\Delta_{23}].
	$$
	Take $\theta=\sum\limits_{i=1}^5\alpha_i\nabla_i\in {\rm H}_{\bf L}^2({\bf L}^{3*}_{03}).$
	The automorphism group of ${\bf L}^{3*}_{03}$ consists of invertible matrices of the form
	$$
	\phi=
	\begin{pmatrix}
	x &    u  &  0\\
	y &  v  & 0 \\
	z &  t  &  xv-yu
	\end{pmatrix}.
	$$

Since
	$$
	\phi^T\begin{pmatrix}
	\alpha_1   &  0    & -2\alpha_4\\
	\alpha_2  &\alpha_3 & -2\alpha_5 \\
	\alpha_4&  \alpha_5& 0 \\
	\end{pmatrix} \phi=\begin{pmatrix}
	\alpha_1^*   &  \alpha^*   & -2\alpha_4^*\\
	\alpha_2^* -\alpha^* &\alpha_3^* & -2\alpha_5^* \\
	\alpha_4^*&  \alpha_5^*& 0
	\end{pmatrix},
	$$
we have that the action of ${\rm Aut} ({\bf L}^{3*}_{03})$ on the subspace
$\langle \sum\limits_{i=1}^5\alpha_i\nabla_i  \rangle$
is given by
$\langle \sum\limits_{i=1}^5\alpha_i^{*}\nabla_i\rangle,$
where
	\[\begin{array}{rcl}\alpha^*_1&=&\alpha _1 x^2+\alpha _2 x y+ \alpha _3 y^2-\alpha _4 x z-\alpha _5y z,\\
	\alpha_2^*&=&2\alpha_1xu+\alpha _2(xv+yu)+2\alpha_3yv-\alpha _4 (xt+zu)- \alpha _5(yt+zv),\\
	\alpha_3^*&=&\alpha _1 u^2+\alpha_2uv+\alpha_3 v^2 -\alpha _4ut-\alpha_5 vt,\\
    \alpha^*_4&=&\left(\alpha _4 x+\alpha _5 y\right)(xv-yu),\\
	\alpha_5^*&=&\left(\alpha_4u+\alpha_5v\right)(xv-yu).\end{array}\]

Since ${\rm H^2_{\bf L}({\bf L}_{03}^{3*})} = {\rm H^2_{\bf N}({\bf L}_{03}^{3*})}\oplus \langle\nabla_4, \nabla_5\rangle$, we have  $(\alpha_4, \alpha_5)\neq(0,0).$
Moreover, without loss of generality, one can assume $\alpha_4\neq0$.
Choosing 
\begin{center}$u=-\frac{\alpha_5v}{\alpha_4},$   $z=\frac{\alpha_1x^2 +\alpha_2xy+\alpha_3y^2}{\alpha_4x+\alpha_5y}$ and $t=\frac{2\alpha_1xu+\alpha_2(xv+yu) +2\alpha_3yv}{\alpha_4x+\alpha_5y},$
\end{center}
we obtain
\begin{center}
$\alpha^*_1=\alpha^*_2=\alpha^*_5=0,$ 
$\alpha_3^*=\frac{(\alpha _1\alpha_5^2-\alpha_2\alpha_4\alpha_5+\alpha_3\alpha_4^2)v^2}{\alpha_4^2}$ and $\alpha^*_4=\frac{v(\alpha _4 x+\alpha _5 y)^2}{\alpha _4}.$
\end{center}

Then we have the following cases.
\begin{enumerate}
\item $\alpha _1\alpha_5^2-\alpha_2\alpha_4\alpha_5+\alpha_3\alpha_4^2 = 0,$ then we have the representative
            $\langle \nabla_4 \rangle$.

    \item $\alpha _1\alpha_5^2-\alpha_2\alpha_4\alpha_5+\alpha_3\alpha_4^2 \neq 0,$ then by choosing $v=\frac{\alpha_4(\alpha _4 x+\alpha _5 y)^2}{(\alpha _1\alpha_5^2-\alpha_2\alpha_4\alpha_5+\alpha_3\alpha_4^2)}$, we have the representative
 $\langle \nabla_3 +\nabla_4\rangle$.
\end{enumerate}

Summarizing, we have the following distinct orbits
\begin{center}
$\langle \nabla_4\rangle$ and
$\langle \nabla_3 +\nabla_4\rangle.$
\end{center}
Hence, we have the following new algebras:

\begin{longtable}{llllllll}
${\bf L}^4_{09}$ &: &
$e_1 e_2 = e_3$&  $e_1 e_3=-2e_4$&  $e_2e_1=-e_3$&  $e_3e_1=e_4$ \\
\hline ${\bf L}^4_{10}$ &: &
$e_1 e_2 = e_3$&$e_1 e_3=-2 e_4$&  $e_2 e_1=-e_3$&$e_2 e_2=e_4$&  $e_3e_1=e_4$
\end{longtable}

\subsubsection{Central extensions of ${\bf L}^{3*}_{04}(\lambda)_{\lambda\neq0}$}
	Let us use the following notations:
	$$
	\nabla_1 = [\Delta_{12}], \quad \nabla_2 =[\Delta_{21}], \quad \nabla_3 =[\Delta_{22}], \quad \nabla_4 =[\Delta_{13}-\Delta_{31}-\Delta_{32}] ,\quad \nabla_5 =[\Delta_{23}+\lambda\Delta_{31}].
	$$
	Take $\theta=\sum\limits_{i=1}^5\alpha_i\nabla_i\in {\rm H}_{\bf L}^2({\bf L}^{3*}_{04}(\lambda)_{\lambda\neq0}).$
	The automorphism group of ${\bf L}^{3*}_{04}(\lambda)_{\lambda\neq0}$ consists of invertible matrices of the form	
	$$
	\phi=
	\begin{pmatrix}
	x &    y  &  0\\
	-\lambda y &  x-y  & 0 \\
	z &  t  &  x^2-xy+\lambda y^2
	\end{pmatrix}.
	$$
	
	Since
	$$
	\phi^T\begin{pmatrix}
	0   &  \alpha_1    & \alpha_4\\
	\alpha_2  &\alpha_3 & \alpha_5 \\
	-\alpha_4+\lambda \alpha_5&  -\alpha_4& 0 \\
	\end{pmatrix} \phi=\begin{pmatrix}
	\lambda \alpha^*   &  \alpha_1^*    & \alpha_4^* \\
	\alpha^*+\alpha_2^*  &\alpha^*+\alpha_3^* & \alpha_5^* \\
	-\alpha_4^*+\lambda \alpha_5^*&  -\alpha_4^*& 0 \\
	\end{pmatrix},
	$$
	we have that the action of ${\rm Aut} ({\bf L}^{3*}_{04}(\lambda)_{\lambda\neq0})$ on the subspace
$\langle \sum\limits_{i=1}^5\alpha_i\nabla_i  \rangle$
is given by
$\langle \sum\limits_{i=1}^5\alpha_i^{*}\nabla_i\rangle,$
where
\[\begin{array}{rcl}\alpha^*_1&=&\alpha_1(x^2-xy)-\alpha_2\lambda y^2-\alpha_3(\lambda xy-\lambda y^2)+\alpha_4(xt-xz)+\alpha_5(\lambda yz-\lambda yt),\\
	\alpha_2^*&=&\alpha_1(xy-\lambda y^2)+\alpha_2x^2-\alpha_3\lambda  xy+\alpha_4(\lambda yt-xt)+\alpha_5\lambda xt,\\
	\alpha_3^*&=&\alpha_1(2xy-y^2)+\alpha_2(2xy-y^2) +\alpha_3\big((x-y)^2-\lambda y^2 \big)\\ & & +\alpha_4(yt-xt-yz)+\alpha_5(\lambda yt+xt-yt-xz+yz),\\
\alpha^*_4&=&(\alpha_4x-\alpha_5\lambda y)(x^2-xy+\lambda  y^2),\\	
\alpha_5^*&=&(\alpha_4y+\alpha_5 (x-y))(x^2-xy+\lambda  y^2).\end{array}\]

Since ${\rm H^2_{\bf L}({\bf L}_{04}^{3*}(\lambda)_{\lambda\neq0})} = {\rm H^2_{\bf N}({\bf L}_{04}^{3*}(\lambda)_{\lambda\neq0})}\oplus \langle\nabla_4, \nabla_5\rangle$, we have  $(\alpha_4, \alpha_5)\neq(0,0).$ Then we have the following cases.	
\begin{enumerate}
\item $\alpha_4=0$, then $\alpha_5\neq0$ and choosing $y=0,$  $t=-\frac{\alpha_2x}{\lambda\alpha_5},$ $z=\frac{(\alpha_3\lambda-\alpha_2)x}{\alpha_5\lambda}$, we have
\begin{center} 
$\alpha^*_2=\alpha_3^*=\alpha_4^*=0,$  $\alpha^*_1=\alpha_1x^2$ and $\alpha_5^*=\alpha_5x^3.$
\end{center}
	
\begin{enumerate}
	
\item $\alpha_1= 0,$ then we have the representative
 $\langle \nabla_5\rangle.$

\item $\alpha_1\neq 0,$ then by choosing $x=\frac{\alpha_1}{\alpha_5}$, we have the representative
 $\langle \nabla_1 +\nabla_5\rangle$.

\end{enumerate}

\item $\alpha_4\neq0,$ and $\alpha_4^2-\alpha_4\alpha_5+\alpha_5^2\lambda\neq0,$ then by choosing $x=\frac{\lambda\alpha_5}{\alpha_4}$ and $y=1$, we have $\alpha_4^*=0$ and it is the situation considered above.

  \item  $\alpha_4\neq0$ and $\alpha_4^2-\alpha_4\alpha_5+\alpha_5^2\lambda=0,$ then by choosing
\begin{center}
$y=0,$  $t= \frac{\alpha_2x}{\alpha_4-\lambda\alpha_5}$ and   $z=\frac{(\alpha_1 \alpha_4-\lambda \alpha_1 \alpha_5 + \alpha_2\alpha_4)x}{\alpha_4(\alpha_4-\lambda\alpha_5)},$
\end{center}
we have
\begin{center}
$\alpha_1^*= \alpha_2^*=0,$  $\alpha_3^*=\frac{\big((\alpha_4-\lambda\alpha_5)(\alpha_3 \alpha_4-\alpha_1 \alpha_5) -\alpha_2\alpha_4^2\big)x^2}{\alpha_4(\alpha_4-\lambda\alpha_5)},$  $\alpha_4^*=x^3 \alpha_4$ and $\alpha_5^*=x^3 \alpha_5$.
\end{center}

\begin{enumerate}
\item  $\alpha_2\alpha_4^2=(\alpha_4-\lambda\alpha_5)(\alpha_3 \alpha_4-\alpha_1 \alpha_5),$ then  we have the representative
 $\langle   \nabla_4+ \frac{1\pm\sqrt{1-4\lambda}}{2\lambda}  \nabla_5\rangle.$

\item $\alpha_2\alpha_4^2\neq(\alpha_4-\lambda\alpha_5)(\alpha_3 \alpha_4-\alpha_1 \alpha_5),$ then  we have the representative
\begin{center}
$\langle  \frac{1}{2\lambda} \nabla_3+  \nabla_4+ \frac{1\pm\sqrt{1-4\lambda}}{2\lambda} \nabla_5\rangle.$
\end{center}
\end{enumerate}
\end{enumerate}

Summarizing, we have the following distinct orbits
\begin{longtable}{lll}

$\langle   \nabla_5 \rangle$ &
$\langle    {2\lambda} \nabla_4+ {(1 -\sqrt{1-4\lambda})}  \nabla_5\rangle$ &
$\langle   \nabla_3+  {2\lambda} \nabla_4+ {(1 -\sqrt{1-4\lambda})}  \nabla_5\rangle$\\
$\langle  \nabla_1+   \nabla_5 \rangle$&
$\langle    {2\lambda} \nabla_4+ {(1 +\sqrt{1-4\lambda})}  \nabla_5\rangle$ &
$\langle  \nabla_3+  {2\lambda} \nabla_4+ {(1+\sqrt{1-4\lambda})}  \nabla_5\rangle.$

\end{longtable}

Hence, we have the following new algebras:
\begin{longtable}{llllllllll}
${\bf L}^4_{11}(\lambda)_{\lambda \neq 0}$ &: &
$e_1 e_1 = \lambda e_3$ &  $e_2 e_1=e_3$&  $e_2e_2=e_3 $\\
&&$e_2 e_3=e_4$ & $e_3 e_1=\lambda e_4$ \\
\hline
${\bf L}^4_{12}(\lambda)_{\lambda \neq 0}$ &: &
$e_1 e_1 = \lambda e_3$& $e_1 e_2=e_4$&  $e_2 e_1=e_3$ \\ &&  $e_2e_2=e_3 $  & $e_2e_3=e_4 $ & $e_3e_1=\lambda e_4 $ \\
\hline
${\bf L}^4_{13}(\lambda)_{\lambda \neq 0}$ &: &
$e_1 e_1 = \lambda e_3$&    $e_2e_2=e_3 $ & $e_2e_3=(1-\sqrt{1-4\lambda})e_4 $ & $e_3e_2=-2\lambda e_4 $ \\ & & $e_2 e_1=e_3$&    $e_1e_3=2\lambda e_4 $  & $e_3e_1=-\lambda(1+\sqrt{1-4\lambda})e_4 $ \\
\hline
${\bf L}^4_{14}(\lambda)_{\lambda \neq 0}$ &: &
$e_1 e_1 = \lambda e_3$&    $e_2e_2=e_3+e_4 $ & $e_2e_3=(1-\sqrt{1-4\lambda})e_4 $ & $e_3e_2=-2\lambda e_4 $ \\ & & $e_2 e_1=e_3$&    $e_1e_3=2\lambda e_4 $  & $e_3e_1=-\lambda(1+\sqrt{1-4\lambda})e_4 $ \\
\hline
${\bf L}^4_{15}(\lambda)_{\lambda \neq 0}$ &: &
$e_1 e_1 = \lambda e_3$&    $e_2e_2=e_3 $ & $e_2e_3=(1+\sqrt{1-4\lambda})e_4 $ & $e_3e_2=-2\lambda e_4 $ \\ & & $e_2 e_1=e_3$&    $e_1e_3=2\lambda e_4 $  & $e_3e_1=-\lambda(1-\sqrt{1-4\lambda})e_4 $ \\
\hline
${\bf L}^4_{16}(\lambda)_{\lambda \neq 0}$ &: &
$e_1 e_1 = \lambda e_3$&    $e_2e_2=e_3 +e_4$ & $e_2e_3=(1+\sqrt{1-4\lambda})e_4 $ & $e_3e_2=-2\lambda e_4 $ \\ & & $e_2 e_1=e_3$&    $e_1e_3=2\lambda e_4 $  & $e_3e_1=-\lambda(1-\sqrt{1-4\lambda})e_4 $ \\
\end{longtable}

	\subsubsection{Central extensions of ${\bf L}^{3*}_{05}$}
	Let us use the following notations:
	$$
	\nabla_1 = [\Delta_{12}], \quad \nabla_2 = [\Delta_{13}-\Delta_{31}], \quad \nabla_3 =[\Delta_{22}+\Delta_{31}], \quad \nabla_4 =[\Delta_{23}].
	$$
	Take $\theta=\sum\limits_{i=1}^4\alpha_i\nabla_i\in {\rm H}_{\bf L}^2({\bf L}^{3*}_{05}).$
	The automorphism group of ${\bf L}^{3*}_{05}$ consists of invertible matrices of the form
	$$
	\phi=
	\begin{pmatrix}
	x &    0  &  0\\
	y &  x^2  & 0 \\
	z & xy  &  x^3
	\end{pmatrix}.
	$$
	
	Since
	$$
	\phi^T\begin{pmatrix}
	0   &  \alpha_1    & \alpha_2\\
	0  &\alpha_3 & \alpha_4 \\
	-\alpha_2+\alpha_3& 0 & 0 \\
	\end{pmatrix} \phi=	\begin{pmatrix}
	\alpha^*   &  \alpha_1^*    & \alpha_2^*\\
	\alpha^{**}  &\alpha_3^* & \alpha_4^* \\
	-\alpha_2^*+\alpha_3^*& 0 & 0 \\
	\end{pmatrix},
	$$
we have that the action of ${\rm Aut} ({\bf L}^{3*}_{05})$ on the subspace $\langle \sum\limits_{i=1}^4\alpha_i\nabla_i  \rangle$
is given by
$\langle \sum\limits_{i=1}^4\alpha_i^{*}\nabla_i\rangle,$
where
\begin{longtable}{ll}	
$\alpha^*_1=\alpha _1 x^3+(\alpha _2 +\alpha _3) x^2y+\alpha _4 xy^2,$&
$\alpha^*_2=\alpha _2x^4+\alpha _4x^3y,$\\
$\alpha_3^*=\alpha _3x^4+\alpha _4x^3 y,$& 
$\alpha_4^*=\alpha _4 x^5.$
\end{longtable}

Since ${\rm H^2_{\bf L}({\bf L}^{3*}_{05})} = {\rm H^2_{\bf N}({\bf L}^{3*}_{05})}\oplus \langle\nabla_3, \nabla_4\rangle$, we have  $(\alpha_3, \alpha_4)\neq(0,0).$ Then we have the following cases.
\begin{enumerate}
\item $\alpha_4=0,$ then
\begin{enumerate}
	\item if $\alpha_2+\alpha_3\neq0,$ then by choosing $y=-\frac{\alpha_1x}{\alpha_2+\alpha_3}$, we have the representative
 $\langle \alpha\nabla_2 +\nabla_3\rangle_{\alpha\neq -1}$;

\item if $ \alpha_2+\alpha_3=0,$ then we have the representatives
 $\langle -\nabla_2 +\nabla_3\rangle$ and $\langle \nabla_1-\nabla_2 +\nabla_3\rangle$  depending on whether $\alpha_1=0$ or not.
\end{enumerate}

\item if $\alpha_4\neq0,$ then by choosing $y=-\frac{\alpha_3x}{\alpha_4},$ we have
\begin{center}
$\alpha^*_1=\frac{(\alpha_1\alpha_4-\alpha_2\alpha_3)x^3}{\alpha_4},$ $\alpha^*_2=(\alpha_2-\alpha_3)x^4,$ 
$\alpha_3^*=0,$ and $\alpha_4^*=\alpha_4x^5,$
\end{center}

\begin{enumerate}
	
\item if $\alpha_1\alpha_4-\alpha_2\alpha_3=0, \ \alpha_2-\alpha_3=0,$ then we have the representative
 $\langle \nabla_4\rangle;$

\item if $\alpha_1\alpha_4-\alpha_2\alpha_3=0, \ \alpha_2-\alpha_3\neq0,$ then we have the representative
 $\langle \nabla_2 +\nabla_4\rangle$;

\item if $\alpha_1\alpha_4-\alpha_2\alpha_3\neq0,$ then we have the representative
 $\langle \nabla_1 + \alpha\nabla_2 +\nabla_4\rangle$.

\end{enumerate}
\end{enumerate}
Hence, we have the following distints orbits
\begin{center}
$\langle \alpha\nabla_2+\nabla_3\rangle,$ 
$\langle \nabla_1 - \nabla_2+\nabla_3\rangle,$ 
$\langle \nabla_4\rangle,$ 
$\langle \nabla_2+ \nabla_4\rangle,$ and $\langle \nabla_1 + \alpha\nabla_2 +\nabla_4\rangle,$
\end{center}
which give the following new algebras:

\begin{longtable}{llllllll}
${\bf L}^4_{17}(\alpha)$ &: &
$e_1 e_1 = e_2$&  $e_1 e_3=\alpha e_4$&$e_2 e_1=e_3$& $e_2 e_2=e_4$& \multicolumn{2}{l}{$e_3e_1=(1-\alpha)e_4$} \\
\hline ${\bf L}^4_{18}$ &: &
$e_1 e_1 = e_2$&$e_1 e_2=e_4$& $e_1 e_3=-e_4$& $e_2 e_1=e_3$&$e_2 e_2=e_4$& $e_3e_1=2e_4$ \\
\hline ${\bf L}^4_{19}$ &: &
$e_1 e_1 = e_2$&  $e_2 e_1=e_3$& $e_2 e_3=e_4$ \\
\hline ${\bf L}^4_{20}$ &: &
$e_1 e_1 = e_2$&  $e_1 e_3=e_4$&$e_2 e_1=e_3$& $e_2 e_3=e_4$& $e_3 e_1=-e_4$ \\
\hline ${\bf L}^4_{21}(\alpha)$ &: &
$e_1 e_1 = e_2$& $e_1 e_2=e_4$&$e_1 e_3=\alpha e_4$&$e_2 e_1=e_3$&$e_2 e_3=e_4$& $e_3 e_1=-\alpha e_4$\\
\end{longtable}

\subsubsection{Central extensions of ${\bf L}^{3*}_{06}(\lambda)$}
	Let us use the following notations:
	\begin{longtable}{lll} $\nabla_1 = [\Delta_{21}]$ &  $\nabla_2 =[(2-\lambda)\Delta_{13}+\lambda\Delta_{22}+\lambda\Delta_{31}]$ & $\nabla_3 = [\Delta_{22}+ \Delta_{13}-\Delta_{31}]$.\end{longtable}
	
Take $\theta=\sum\limits_{i=1}^3\alpha_i\nabla_i\in {\rm H}_{\bf L}^2({\bf L}^{3*}_{06}(\lambda)).$ 	The automorphism group of ${\bf L}^{3*}_{06}(\lambda)$ consists of invertible matrices of the form
	$$
	\phi=
	\begin{pmatrix}
	x &    0  & 0\\
	y &  x^2  & 0 \\
	z & xy(1+\lambda)  &  x^3
	\end{pmatrix}.
	$$
	
	Since
	$$
	\phi^T\begin{pmatrix}
	0   &  0    & (2-\lambda)\alpha_2+\alpha_3\\
	\alpha_1  &\lambda\alpha_2+\alpha_3 & 0 \\
	\lambda\alpha_2-\alpha_3& 0 & 0 \\
	\end{pmatrix} \phi=	\begin{pmatrix}
	\alpha^{**}   &  \alpha^{*}    & (2-\lambda)\alpha_2^*+\alpha_3^*\\
	\alpha_1^*+\lambda\alpha^{*}  &\lambda\alpha_2^*+\alpha_3^* & 0 \\
	\lambda\alpha_2^*-\alpha_3^*& 0 & 0 \\
	\end{pmatrix},
	$$
we have that the action of ${\rm Aut} ({\bf L}^{3*}_{06}(\lambda))$ on the subspace $\langle \sum\limits_{i=1}^3\alpha_i\nabla_i  \rangle$
is given by
$\langle \sum\limits_{i=1}^3\alpha_i^{*}\nabla_i\rangle,$
where	
\begin{longtable}{lll}	
$\alpha^*_1=\alpha _1 x^3 +\big(\alpha _2(\lambda^3-\lambda ^2) -\alpha _3 (\lambda^2+3\lambda)\big) x^2y$& 
$\alpha^*_2=\alpha _2 x^4$&
$\alpha_3^*=\alpha _3 x^4.$\end{longtable}
	
Since ${\rm H^2_{\bf L}}({\bf L}^{3*}_{06}(\lambda)) = {\rm H^2_{\bf N}}({\bf L}^{3*}_{06}(\lambda))\oplus \langle \nabla_3\rangle$, we have  $\alpha_3\neq 0.$ Then we have the following cases	
\begin{enumerate}
\item if $\lambda=0,$ then we have the representatives $\langle \alpha \nabla_2 +\nabla_3\rangle$ and $\langle \nabla_1+\alpha \nabla_2 +\nabla_3\rangle$ depending on whether $\alpha_1=0$ or not.

\item if $\lambda=1,$ then by choosing $y=\frac{\alpha_1 x}{4\alpha_3},$ we have the representative $\langle \alpha \nabla_2 +\nabla_3\rangle$.

\item if $\lambda\notin\{0;1\}$, then:
 \begin{enumerate}
 \item if  $\alpha _2(\lambda^3-\lambda ^2) -\alpha _3 (\lambda^2+3\lambda)\neq 0,$ then by choosing $y=\frac{\alpha_1 x}{\alpha_3(\lambda^2+3\lambda) - \alpha_2(\lambda^3-\lambda_2)},$ we have the representative $\langle \alpha \nabla_2 +\nabla_3\rangle_{\alpha\neq \frac{\lambda+3}{\lambda(\lambda-1)}}$;
 \item if $\alpha _2(\lambda^3-\lambda ^2) -\alpha _3 (\lambda^2+3\lambda)= 0,$ then  we have the representatives \begin{center}
     $\langle \frac{\lambda+3}{\lambda(\lambda-1)} \nabla_2 +\nabla_3\rangle$ and $\langle \frac{1}{\lambda(\lambda-1)}\nabla_1+\frac{\lambda+3}{\lambda(\lambda-1)}\nabla_2+\nabla_3\rangle,$
     \end{center} depending on whether $\alpha_1=0$ or not.
\end{enumerate}
\end{enumerate}

Thus, we have the following orbits:
\begin{center}  
$\langle \nabla_1+\alpha\nabla_2+\nabla_3\rangle_{\lambda=0},$  $\langle\alpha\nabla_2+\nabla_3\rangle,$ and  
$\langle \nabla_1+(\lambda+3)\nabla_2+\lambda(\lambda-1)\nabla_3\rangle_{\lambda\neq 0;1}$.\end{center}
Hence, we have the following new algebras:
\small
\begin{longtable}{lllllllllll}
${\bf L}^4_{22}(\alpha)$ &: &$e_1 e_1 = e_2$&$e_1 e_2=e_3$& $e_1 e_3=(2\alpha+1) e_4$ \\ & &$e_2 e_1=e_4$&$e_2 e_2=e_4$ &$e_3 e_1=-e_4$  \\ \hline
${\bf L}^4_{23}(\lambda,\alpha)$ &: &
$e_1 e_1 = e_2$&$e_1 e_2=e_3$& $e_1 e_3=\big((2-\lambda)\alpha+1\big) e_4$ \\ && $e_2 e_1=\lambda e_3$& $e_2 e_2=(\lambda\alpha+1)e_4$&$e_3 e_1=(\lambda\alpha-1)e_4$ \\ \hline
${\bf L}^4_{24}(\lambda)_{\lambda \notin\{0;1\}}$ &: &
$e_1 e_1 = e_2$&$e_1 e_2=e_3$& $e_1 e_3=2(3-\lambda)e_4$ \\ & & $e_2 e_1=\lambda e_3+e_4$&
$e_2 e_2=2\lambda(\lambda+1)e_4$&$e_3 e_1=4\lambda e_4$ \\
\end{longtable}

\subsection{Classification theorem}
Now we are ready summarize all results related to the algebraic classification of complex $4$-dimensional nilpotent left-symmetric  algebras.

\begin{theoremA}
Let ${\bf L}$ be a complex  $4$-dimensional nilpotent left-symmetric algebra.
Then either ${\bf L}$ is a Novikov algebra or   it is isomorphic to one algebra from the following list:

\begin{longtable}{lllllllllll}
${\bf L}^4_{01}$ &$:$ &
$e_1 e_1 = e_2$&  $e_1 e_2=e_4$& $e_2 e_3=e_4$& $e_3e_1=e_4$ \\
\hline${\bf L}^4_{02}$ &$:$ &
$e_1 e_1 = e_2$&  $e_2 e_3=e_4$& $e_3e_1=e_4$ \\
\hline${\bf L}^4_{03}$ &$:$ &
$e_1 e_1 = e_2$&  $e_1 e_2=e_4$& $e_2 e_3=e_4$ \\
\hline${\bf L}^4_{04}$ &$:$ &
$e_1 e_1 = e_2$& $e_2 e_3=e_4$\\
\hline${\bf L}^4_{05}$ &$:$ &
$e_1 e_1 = e_3$&  $e_2 e_2=e_3$&  $e_3e_1=e_4$ \\
\hline ${\bf L}^4_{06}$ &$:$ &
$e_1 e_1 = e_3$&  $e_1 e_2=e_4$&$e_2 e_2=e_3$& $e_3e_1=e_4$ \\
\hline${\bf L}^4_{07}$ &$:$ &
$e_1 e_1 = e_3$&  $e_2 e_2=e_3$& $e_3 e_1=e_4$&$e_3 e_2= ie_4$ \\
\hline ${\bf L}^4_{08}$ &$:$ &
$e_1 e_1 = e_3$& $e_1 e_2=e_4$& $e_2 e_2=e_3$& $e_3 e_1=e_4$&$e_3 e_2= ie_4$\\
\hline
${\bf L}^4_{09}$ &$:$ &
$e_1 e_2 = e_3$&  $e_1 e_3=-2e_4$&  $e_2e_1=-e_3$&  $e_3e_1=e_4$ \\
\hline ${\bf L}^4_{10}$ &$:$ &
$e_1 e_2 = e_3$&$e_1 e_3=-2 e_4$&  $e_2 e_1=-e_3$&$e_2 e_2=e_4$&  $e_3e_1=e_4$\\
\hline ${\bf L}^4_{11}(\lambda)_{\lambda \neq 0}$ &$:$ &
$e_1 e_1 = \lambda e_3$ &  $e_2 e_1=e_3$&  $e_2e_2=e_3 $&$e_2 e_3=e_4$ & $e_3 e_1=\lambda e_4$ \\
\hline
${\bf L}^4_{12}(\lambda)_{\lambda \neq 0}$ &$:$ &
$e_1 e_1 = \lambda e_3$& $e_1 e_2=e_4$&  $e_2 e_1=e_3$&\\
&&$e_2e_2=e_3 $ & $e_2e_3=e_4 $ & $e_3e_1=\lambda e_4 $ \\
\hline
${\bf L}^4_{13}(\lambda)_{\lambda \neq 0}$ &$:$ &
$e_1 e_1 = \lambda e_3$&    $e_2e_2=e_3 $ & \multicolumn{2}{l}{$e_2e_3=(1-\sqrt{1-4\lambda})e_4 $} & $e_3e_2=-2\lambda e_4 $ \\ & & $e_2 e_1=e_3$&    $e_1e_3=2\lambda e_4 $  & \multicolumn{2}{l}{$e_3e_1=-\lambda(1+\sqrt{1-4\lambda})e_4 $} \\
\hline
${\bf L}^4_{14}(\lambda)_{\lambda \neq 0}$ &$:$ &
$e_1 e_1 = \lambda e_3$&    $e_2e_2=e_3+e_4 $ & 
\multicolumn{2}{l}{$e_2e_3=(1-\sqrt{1-4\lambda})e_4 $} & $e_3e_2=-2\lambda e_4 $ \\ & & $e_2 e_1=e_3$&    $e_1e_3=2\lambda e_4 $  & \multicolumn{2}{l}{$e_3e_1=-\lambda(1+\sqrt{1-4\lambda})e_4 $} \\
\hline
${\bf L}^4_{15}(\lambda)_{\lambda \neq 0}$ &$:$ &
$e_1 e_1 = \lambda e_3$&    $e_2e_2=e_3 $ & 
\multicolumn{2}{l}{$e_2e_3=(1+\sqrt{1-4\lambda})e_4 $} & $e_3e_2=-2\lambda e_4 $ \\ & & $e_2 e_1=e_3$&    $e_1e_3=2\lambda e_4 $  & \multicolumn{2}{l}{$e_3e_1=-\lambda(1-\sqrt{1-4\lambda})e_4 $} \\
\hline
${\bf L}^4_{16}(\lambda)_{\lambda \neq 0}$ &$:$ &
$e_1 e_1 = \lambda e_3$&    $e_2e_2=e_3 +e_4$ & 
\multicolumn{2}{l}{$e_2e_3=(1+\sqrt{1-4\lambda})e_4 $} & $e_3e_2=-2\lambda e_4 $ \\ & & $e_2 e_1=e_3$&    $e_1e_3=2\lambda e_4 $  & \multicolumn{2}{l}{$e_3e_1=-\lambda(1-\sqrt{1-4\lambda})e_4 $} \\
\hline
${\bf L}^4_{17}(\alpha)$ &$:$ &
$e_1 e_1 = e_2$&  $e_1 e_3=\alpha e_4$&$e_2 e_1=e_3$& $e_2 e_2=e_4$& $e_3e_1=(1-\alpha)e_4$ \\
\hline ${\bf L}^4_{18}$ &$:$ &
$e_1 e_1 = e_2$&$e_1 e_2=e_4$& $e_1 e_3=-e_4$&\\
&&$e_2 e_1=e_3$&$e_2 e_2=e_4$& $e_3e_1=2e_4$ \\
\hline ${\bf L}^4_{19}$ &$:$ &
$e_1 e_1 = e_2$&  $e_2 e_1=e_3$& $e_2 e_3=e_4$ \\
\hline ${\bf L}^4_{20}$ &$:$ &
$e_1 e_1 = e_2$&  $e_1 e_3=e_4$&$e_2 e_1=e_3$& $e_2 e_3=e_4$& $e_3 e_1=-e_4$ \\
\hline ${\bf L}^4_{21}(\alpha)$ &$:$ &
$e_1 e_1 = e_2$& $e_1 e_2=e_4$&$e_1 e_3=\alpha e_4$&\\
&&$e_2 e_1=e_3$&$e_2 e_3=e_4$& $e_3 e_1=-\alpha e_4$\\
\hline
${\bf L}^4_{22}(\alpha)$ &$:$ &
$e_1 e_1 = e_2$&$e_1 e_2=e_3$& \multicolumn{2}{l}{$e_1 e_3=(2\alpha+1) e_4$}  \\
&&$e_2 e_1=e_4$&$e_2 e_2=e_4$ &$e_3 e_1=-e_4$  \\ \hline
${\bf L}^4_{23}(\lambda,\alpha)$ &$:$ &
$e_1 e_1 = e_2$&$e_1 e_2=e_3$& \multicolumn{2}{l}{$e_1 e_3=\big((2-\lambda)\alpha+1\big) e_4$} \\
&& $e_2 e_1=\lambda e_3$& $e_2 e_2=(\lambda\alpha+1)e_4$&
\multicolumn{2}{l}{$e_3 e_1=(\lambda\alpha-1)e_4$} \\ \hline
${\bf L}^4_{24}(\lambda)_{\lambda \notin\{0;1\}}$ &$:$ &
$e_1 e_1 = e_2$&$e_1 e_2=e_3$& 
\multicolumn{2}{l}{$e_1 e_3=2(3-\lambda)e_4$}  \\
&& $e_2 e_1=\lambda e_3+e_4$&
$e_2 e_2=2\lambda(\lambda+1)e_4$&$e_3 e_1=4\lambda e_4$ \\

\end{longtable}

\end{theoremA}

\section{The geometric classification of nilpotent left-symmetric algebras}

\subsection{Definitions and notation}
Given an $n$-dimensional vector space $\mathbb V$, the set ${\rm Hom}(\mathbb V \otimes \mathbb V,\mathbb V) \cong \mathbb V^* \otimes \mathbb V^* \otimes \mathbb V$
is a vector space of dimension $n^3$. This space has the structure of the affine variety $\mathbb{C}^{n^3}.$ Indeed, let us fix a basis $e_1,\dots,e_n$ of $\mathbb V$. Then any $\mu\in {\rm Hom}(\mathbb V \otimes \mathbb V,\mathbb V)$ is determined by $n^3$ structure constants $c_{ij}^k\in\mathbb{C}$ such that
$\mu(e_i\otimes e_j)=\sum\limits_{k=1}^nc_{ij}^ke_k$. A subset of ${\rm Hom}(\mathbb V \otimes \mathbb V,\mathbb V)$ is {\it Zariski-closed} if it can be defined by a set of polynomial equations in the variables $c_{ij}^k$ ($1\le i,j,k\le n$).

Let $T$ be a set of polynomial identities.
The set of algebra structures on $\mathbb V$ satisfying polynomial identities from $T$ forms a Zariski-closed subset of the variety ${\rm Hom}(\mathbb V \otimes \mathbb V,\mathbb V)$. We denote this subset by $\mathbb{L}(T)$.
The general linear group $GL(\mathbb V)$ acts on $\mathbb{L}(T)$ by conjugation:
$$ (g * \mu )(x\otimes y) = g\mu(g^{-1}x\otimes g^{-1}y)$$
for $x,y\in \mathbb V$, $\mu\in \mathbb{L}(T)\subset {\rm Hom}(\mathbb V \otimes\mathbb V, \mathbb V)$ and $g\in GL(\mathbb V)$.
Thus, $\mathbb{L}(T)$ is decomposed into $GL(\mathbb V)$-orbits that correspond to the isomorphism classes of algebras.
Let $O(\mu)$ denote the orbit of $\mu\in\mathbb{L}(T)$ under the action of $GL(\mathbb V)$ and $\overline{O(\mu)}$ denote the Zariski closure of $O(\mu)$.

Let $\mathcal A$ and $\mathcal B$ be two $n$-dimensional algebras satisfying the identities from $T$, and let $\mu,\lambda \in \mathbb{L}(T)$ represent $\mathcal A$ and $\mathcal B$, respectively.
We say that $\mathcal A$ degenerates to $\mathcal B$ and write $\mathcal A\to \mathcal B$ if $\lambda\in\overline{O(\mu)}$.
Note that in this case we have $\overline{O(\lambda)}\subset\overline{O(\mu)}$. Hence, the definition of a degeneration does not depend on the choice of $\mu$ and $\lambda$. If $\mathcal A\not\cong \mathcal B$, then the assertion $\mathcal A\to \mathcal B$ is called a {\it proper degeneration}. We write $\mathcal A\not\to \mathcal B$ if $\lambda\not\in\overline{O(\mu)}$.

Let $\mathcal A$ be represented by $\mu\in\mathbb{L}(T)$. Then  $\mathcal A$ is  {\it rigid} in $\mathbb{L}(T)$ if $O(\mu)$ is an open subset of $\mathbb{L}(T)$.
 Recall that a subset of a variety is called irreducible if it cannot be represented as a union of two non-trivial closed subsets.
 A maximal irreducible closed subset of a variety is called an {\it irreducible component}.
It is well known that any affine variety can be represented as a finite union of its irreducible components in a unique way.
The algebra $\mathcal A$ is rigid in $\mathbb{L}(T)$ if and only if $\overline{O(\mu)}$ is an irreducible component of $\mathbb{L}(T)$.

Given the spaces $U$ and $W$, we write simply $U>W$ instead of $\dim \,U>\dim \,W$.



\subsection{Method of the description of  degenerations of algebras}

In the present work we use the methods applied to Lie algebras in \cite{BC99,GRH,GRH2,S90}.
First of all, if $\mathcal A\to \mathcal B$ and $\mathcal A\not\cong \mathcal B$, then $\mathfrak{Der}(\mathcal A)<\mathfrak{Der}(\mathcal B)$, where $\mathfrak{Der}(\mathcal A)$ is the Lie algebra of derivations of $\mathcal A$. We compute the dimensions of algebras of derivations and check the assertion $\mathcal A\to \mathcal B$ only for   $\mathcal A$ and $\mathcal B$ such that $\mathfrak{Der}(\mathcal A)<\mathfrak{Der}(\mathcal B)$.


To prove degenerations, we construct families of matrices parametrized by $t$. Namely, let $\mathcal A$ and $\mathcal B$ be two algebras represented by the structures $\mu$ and $\lambda$ from $\mathbb{L}(T)$ respectively. Let $e_1,\dots, e_n$ be a basis of $\mathbb  V$ and $c_{ij}^k$ ($1\le i,j,k\le n$) be the structure constants of $\lambda$ in this basis. If there exist $a_i^j(t)\in\mathbb{C}$ ($1\le i,j\le n$, $t\in\mathbb{C}^*$) such that $E_i^t=\sum\limits_{j=1}^na_i^j(t)e_j$ ($1\le i\le n$) form a basis of $\mathbb V$ for any $t\in\mathbb{C}^*$, and the structure constants of $\mu$ in the basis $E_1^t,\dots, E_n^t$ are such rational functions $c_{ij}^k(t)\in\mathbb{C}(t)$ that $c_{ij}^k(0)=c_{ij}^k$, then $\mathcal A\to \mathcal B$.
In this case  $E_1^t,\dots, E_n^t$ is called a {\it parametrized basis} for $\mathcal A\to \mathcal B$.
To simplify our equations, we will use the notation $A_i=\langle e_i,\dots,e_n\rangle,\ i=1,\ldots,n$ and write simply $A_pA_q\subset A_r$ instead of $c_{ij}^k=0$ ($i\geq p$, $j\geq q$, $k< r$).

Since the variety of $4$-dimensional nilpotent left-symmetric algebras  contains infinitely many non-isomorphic algebras, we have to do some additional work.
Let $\mathcal A(*):=\{\mathcal A(\alpha)\}_{\alpha\in I}$ be a series of algebras, and let $\mathcal B$ be another algebra. Suppose that for $\alpha\in I$, $\mathcal A(\alpha)$ is represented by the structure $\mu(\alpha)\in\mathbb{L}(T)$ and $B\in\mathbb{L}(T)$ is represented by the structure $\lambda$. Then we say that $\mathcal A(*)\to \mathcal B$ if 
$\lambda\in\overline{ \bigcup_{\alpha\in I}O(\mu(\alpha))}$, and $\mathcal A(*)\not\to \mathcal B$ if 
$\lambda\not\in\overline{\bigcup_{\alpha\in I}O(\mu(\alpha))}$.

Let $\mathcal A(*)$, $\mathcal B$, $\mu(\alpha)$ ($\alpha\in I$) and $\lambda$ be as above. To prove $\mathcal A(*)\to \mathcal B$ it is enough to construct a family of pairs $(f(t), g(t))$ parametrized by $t\in\mathbb{C}^*$, where $f(t)\in I$ and $g(t)\in GL(\mathbb V)$. Namely, let $e_1,\dots, e_n$ be a basis of $\mathbb V$ and $c_{ij}^k$ ($1\le i,j,k\le n$) be the structure constants of $\lambda$ in this basis. If we construct $a_i^j:\mathbb{C}^*\to \mathbb{C}$ ($1\le i,j\le n$) and $f: \mathbb{C}^* \to I$ such that $E_i^t=\sum\limits_{j=1}^na_i^j(t)e_j$ ($1\le i\le n$) form a basis of $\mathbb V$ for any  $t\in\mathbb{C}^*$, and the structure constants of $\mu_{f(t)}$ in the basis $E_1^t,\dots, E_n^t$ are such rational functions $c_{ij}^k(t)\in\mathbb{C}(t)$ that $c_{ij}^k(0)=c_{ij}^k$, then $\mathcal A(*)\to \mathcal B$. In this case  $E_1^t,\dots, E_n^t$ and $f(t)$ are called a parametrized basis and a {\it parametrized index} for $\mathcal A(*)\to \mathcal B$, respectively.

We now explain how to prove $\mathcal A(*)\not\to\mathcal  B$.
Note that if $\mathfrak{Der} \ \mathcal A(\alpha)  > \mathfrak{Der} \  \mathcal B$ for all $\alpha\in I,$ then $\mathcal A(*)\not\to\mathcal B$.
One can also use the following  Lemma, whose proof is the same as the proof of Lemma 1.5 from \cite{GRH}.

\begin{lemma}\label{gmain}
Let $\mathfrak{B}$ be a Borel subgroup of $GL(\mathbb V)$ and $\mathcal{R}\subset \mathbb{L}(T)$ be a $\mathfrak{B}$-stable closed subset.
If $\mathcal A(*) \to \mathcal B$ and for any $\alpha\in I$ the algebra $\mathcal A(\alpha)$ can be represented by a structure $\mu(\alpha)\in\mathcal{R}$, then there is $\lambda\in \mathcal{R}$ representing $\mathcal B$.
\end{lemma}

\subsection{The geometric classification of $4$-dimensional nilpotent left-symmetric algebras}
The main result of the present section is the following theorem.

\begin{theoremB}\label{geobl}
The variety of $4$-dimensional nilpotent left-symmetric algebras  has 
dimension $15$ and it has 
three irreducible components
defined by infinite families of algebras 
${\bf L}^{4}_{12}(\lambda),$   ${\bf L}^{4}_{21}(\lambda)$ and ${\bf L}^{4}_{23}(\lambda,\alpha).$  
\end{theoremB}

\begin{Proof}
Recall that the description of all irreducible components  of $4$-dimensional nilpotent Novikov algebras was given in \cite{kkk18}.
Using the cited result, we can see that the variety of $4$-dimensional Novikov algebras has two irreducible components given by the following
families of algebras:

\begin{longtable}{lllllllll}
${\mathcal N}^4_{20}(\alpha)$  & $:$ & $e_1e_2 = e_3$ & $e_1e_1 = \alpha e_4$ & $e_1e_3 = e_4$ & $e_2e_2 = e_4$ & $e_2e_3 = e_4$ & $e_3e_2 = -e_4$ \\
\hline
${\mathcal N}^4_{22}(\lambda)$ & $:$ & $e_1e_1 = e_2$ & $e_1e_2 = e_3$ & $e_1e_3 = (2-\lambda)e_4$ & $e_2e_1 = \lambda e_4$ & $e_2e_2 = \lambda e_4$ & $e_3e_1 = \lambda e_4$
\end{longtable}

Now we can prove that the variety of $4$-dimensional nilpotent left-symmetric algebras has three irreducible components.
One can easily compute that
\begin{longtable}{lllll} $\mathfrak{Der} \ {\bf L}^{4}_{12}(\lambda)=2,$ 
& $\mathfrak{Der} \ {\bf L}^{4}_{21}(\lambda)=2,$ &  $\mathfrak{Der} \ {\bf L}^{4}_{23}(\lambda, \alpha)=3.$ 
\end{longtable}

The list of all necessary degenerations is given below:

{\tiny
\begin{longtable}{|lcl|llll|}


\hline
${\bf L}^{4}_{12}(t)$&$\to$&${\mathcal N}^{4}_{20}(\alpha)$  & 
\multicolumn{2}{l}{$E_1^t=\frac{1}{\alpha t-\alpha t^2}e_1+\frac{1}{\alpha (t-1)}e_2+\frac{1+\alpha-2 t-\alpha t}{\alpha^2 (t-1) t^2}e_3$} & 
\multicolumn{2}{l|}{$E_2^t=\frac{1}{\alpha (t-1)}e_2+\frac{1}{\alpha^2 t-\alpha^2 t^2}e_3$} \\
&&& \multicolumn{2}{l}{$E_3^t=\frac{1}{\alpha^2 (t-1) t}e_3-\frac{1+\alpha}{\alpha^3 (t-1) t^2}e_4$}& \multicolumn{2}{l|}{$ E_4^t=\frac{1}{\alpha^3 (t-1)^2 t}e_4$} \\\hline
\multicolumn{7}{|l|}{${\bf L}^{4}_{23}\left(t+\lambda ,-\frac{2+t}{t^2+t \lambda }\right)  \to {\mathcal N}^4_{22}(\lambda)$}\\  
\multicolumn{3}{|l}{}& 
$E_1^t=t e_1$ & $E_2^t=t^{2}e_2$ &$E_3^t=t^3e_3$& $ E_4^t=-\frac{2 t^3}{2 t+\lambda }e_4$ \\\hline
${\bf L}^{4}_{21}(-t)$&$\to$&${\bf L}^{4}_{01}$  & 
\multicolumn{2}{l}{$E_1^t=e_1+t^2e_2+(1-2 t^2+t^3)e_3$} & \multicolumn{2}{l|}{$E_2^t=t^2e_2+t^2e_3+t e_4$} \\
&&& \multicolumn{2}{l}{$E_3^t=t e_3+(1-t^2+t^3)e_4$}& \multicolumn{2}{l|}{$ E_4^t=t^3 e_4$} \\\hline
${\bf L}^{4}_{01}$&$\to$&${\bf L}^{4}_{02}$  & 
$E_1^t=e_1$ & $E_2^t=e_2$ &$E_3^t=t^{-1}e_3$& $ E_4^t=t^{-1}e_4$ \\\hline
${\bf L}^{4}_{01}$&$\to$&${\bf L}^{4}_{03}$  & 
$E_1^t=t^{-1}e_1$ & $E_2^t=t^{-2}e_2$ &$E_3^t=t^{-1}e_3$& $ E_4^t=t^{-3}e_4$ \\\hline
${\bf L}^{4}_{01}$&$\to$&${\bf L}^{4}_{04}$  & 
$E_1^t=t^{-1} e_1$ & $E_2^t=t^{-2}e_2$ &$E_3^t=t^{-2}e_3$& $ E_4^t=t^{-4}e_4$ \\\hline
${\bf L}^{4}_{06}$&$\to$&${\bf L}^{4}_{05}$  & 
$E_1^t=t^{-1}e_1$ & $E_2^t=t^{-1}e_2$ &$E_3^t=t^{-2}e_3$& $ E_4^t=t^{-3}e_4$ \\\hline
${\bf L}^{4}_{12}(t^{-2})$&$\to$&${\bf L}^{4}_{06}$  & 
$E_1^t=t^{3}e_1$ & $E_2^t=t^{2}e_2$ &$E_3^t=t^{4}e_3$& $ E_4^t=t^{5}e_4$ \\\hline
${\bf L}^{4}_{08}$&$\to$&${\bf L}^{4}_{07}$  & 
$E_1^t=t^{-1}e_1$ & $E_2^t=t^{-1}e_2$ &$E_3^t=t^{-2}e_3$& $ E_4^t=t^{-3}e_4$ \\\hline

${\bf L}^{4}_{16}(\frac{1-i t}{t^2})$&$\to$&${\bf L}^{4}_{08}$  & 
\multicolumn{3}{l}{$E_1^t=\frac{2 t^3 \left(2 i+2 t-i t^2\right)}{(i+t)^2 (2 i+t)^3}e_1+\frac{i t^3 \left(t^2+2 i t-2\right)}{(i+t)^2 (2 i+t)^3}e_2-\frac{i t^4 \left(t^2+2 i t-2\right)^2}{(i+t)^4 (2 i+t)^4}e_3$} & \multicolumn{1}{l|}{$E_2^t=-\frac{t^2 \left(t^2+2 i t-2\right)}{\left(t^2+3 i t-2\right)^2}e_2$} \\
&&& \multicolumn{3}{l}{$E_3^t=\frac{t^4 \left(t^2+2 i t-2\right)^2}{(i+t)^4 (2 i+t)^4}e_3+\frac{t^4 \left(t^2+2 i t-2\right)^2}{(i+t)^4 (2 i+t)^4}e_4$}& \multicolumn{1}{l|}{$ E_4^t=-\frac{2 t^4 \left(t^2+2 i t-2\right)^3}{(i+t)^5 (2 i+t)^6}e_4$} \\\hline
${\bf L}^{4}_{10}$&$\to$&${\bf L}^{4}_{09}$  & 
$E_1^t=t^{-1}e_1$ & $E_2^t=e_2$ &$E_3^t=t^{-1}e_3$& $ E_4^t=t^{-2}e_4$ \\\hline
${\bf L}^{4}_{23}(-1,\frac{1+t}{1-t})$&$\to$&${\bf L}^{4}_{10}$  & 
$E_1^t=t e_1$ & $E_2^t=te_2$ &$E_3^t=t^2e_3$& $ E_4^t=\frac{2 t^3}{-1+t}e_4$ \\\hline
${\bf L}^{4}_{12}(\lambda)$&$\to$&${\bf L}^{4}_{11}(\lambda)$  & 
$E_1^t=t^{-1}e_1$ & $E_2^t=t^{-1}e_2$ &$E_3^t=t^{-2}e_3$& $ E_4^t=t^{-3}e_4$ \\\hline
${\bf L}^{4}_{14}(\lambda)$&$\to$&${\bf L}^{4}_{13}(\lambda)$  & 
$E_1^t=t^{-1}e_1$ & $E_2^t=t^{-1}e_2$ &$E_3^t=t^{-2}e_3$& $ E_4^t=t^{-3}e_4$ \\\hline

${\bf L}^{4}_{12}(t+\lambda)$ & $\to$ & \multicolumn{1}{l}{${\bf L}^{4}_{14}(\lambda)$}&
\multicolumn{4}{l|}{$\gamma=\sqrt{1-4\lambda}$}\\
\multicolumn{7}{|l|}{
$E_1^t=-\frac{t \left(1+\gamma\right)^2}{2 \lambda ^2}e_1+
\frac{t \left(1+\gamma\right)}{\lambda }e_2+\frac{ t \left(t^2(t+\lambda )\left(1+\gamma\right)^4+4\lambda^2\left(1+\gamma\right)^2\left(t(1-3 \lambda )-\lambda^2 \right)+32t^2 \lambda^3\right)}{4\lambda ^3 (t+\lambda )^2 \left(t \left(1+\gamma\right)-\left(1-\gamma\right) \lambda \right)}e_3$} \\
\multicolumn{7}{|l|}{
 $E_2^t=-\frac{t \left(1-\gamma\right)}{\lambda ^2}e_1+
\frac{2 t \left((1-\gamma)(\lambda^2-t)+t\lambda(3-\gamma)\right)}{\lambda^3(1-\gamma)}e_2+\frac{t\Big[t^2\left(2t(1-\gamma )^4-\lambda(1-6 \gamma ^2+8 \gamma ^3-3 \gamma ^4-32 \lambda ^2)\right)+ (1-\gamma )^2\lambda ^2\left( t(5+2 \gamma +\gamma ^2-12 \lambda)-4\lambda ^2 \right)\Big]}{2\lambda ^3(1-\gamma) (t+\lambda )^2 \left(t \left(1-\gamma\right)-\left(1+\gamma\right) \lambda \right)}e_3$}\\
\multicolumn{5}{|l}{$E_3^t=\frac{t^3 \left(1+\gamma\right)^4 }{4\lambda ^5}e_3-\frac{t^3 \left(1+\gamma\right)^3 \left(t(t+\lambda)\left(1+\gamma\right)^2 +4 \lambda^2 (1-\lambda -t)\right)}{8 \lambda ^6 (t+\lambda )^2}e_4$} &
\multicolumn{2}{l|}{$E_4^t=\frac{t^4 \left(1+\gamma\right)^5 }{8 \lambda ^7}e_4$} \\
 \hline

${\bf L}^{4}_{16}(\lambda)$&$\to$&${\bf L}^{4}_{15}(\lambda)$  & 
$E_1^t=t^{-1}e_1$ & $E_2^t=t^{-1}e_2$ &$E_3^t=t^{-2}e_3$& $ E_4^t=t^{-3}e_4$ \\\hline

${\bf L}^{4}_{12}(t+\lambda)$ & $\to$ & \multicolumn{1}{l} {${\bf L}^{4}_{16}(\lambda)$} &
 \multicolumn{4}{l|}{ 
$\gamma=\sqrt{1-4\lambda}$} \\
\multicolumn{7}{|l|}{$E_1^t=-\frac{t \left(1-\gamma\right)^2}{2\lambda ^2}e_1+\frac{t(1- \gamma)}{\lambda }e_2+\frac{ t \left(t^2(t+\lambda )\left(1-\gamma\right)^4+4 \lambda ^2 \left( 1-\gamma\right)^2\left(t(1-3 \lambda )-\lambda ^2\right)+32t^2\lambda ^3\right)}{4\lambda ^3 (t+\lambda )^2\left( t (1-\gamma)-(1+\gamma) \lambda \right)}e_3$} \\
\multicolumn{7}{|l|}{$E_2^t=-\frac{t \left(1-\gamma\right)}{\lambda ^2}e_1+
\frac{2 t \left((1-\gamma)(\lambda^2-t)+t\lambda(3-\gamma)\right)}{\lambda^3(1-\gamma)}e_2+\frac{t\Big[t^2\left(2t(1-\gamma )^4-\lambda(1-6 \gamma ^2+8 \gamma ^3-3 \gamma ^4-32 \lambda ^2)\right)+ (1-\gamma )^2\lambda ^2\left( t(5+2 \gamma +\gamma ^2-12 \lambda)-4\lambda ^2 \right)\Big]}{2\lambda ^3(1-\gamma) (t+\lambda )^2 \left(t \left(1-\gamma\right)-\left(1+\gamma\right) \lambda \right)}e_3$} \\
\multicolumn{5}{|l}{
$E_3^t=
\frac{t^3 \left(1-\gamma\right)^4}{4\lambda ^5}
e_3-\frac{t^3 \left(1-\gamma\right)^3 \left(t(t+\lambda)(1-\gamma)^2 +4\lambda^2(1-\lambda-t )\right)}{8 \lambda ^6 (t+\lambda )^2}e_4$} &
\multicolumn{2}{l|}{$ E_4^t=\frac{t^4 \left(1-\gamma\right)^5 }{8 \lambda ^7}e_4$}
\\ 

\hline

\multicolumn{7}{|l|}{${\bf L}^{4}_{23}\left(\frac{1}{t^3},\frac{t^3-t^6 (\alpha-1)}{1+t^3 (\alpha-1)}\right) \to {\bf L}^{4}_{17}(\alpha)$}\\  
\multicolumn{3}{|l}{}& 
$E_1^t=t e_1$ & $E_2^t=t^{2}e_2$ &$E_3^t=e_3$& $ E_4^t=\frac{2 t^4}{1+t^3 (-1+\alpha )}e_4$ \\\hline
${\bf L}^{4}_{21}(\frac{2}{1+t})$&$\to$&${\bf L}^{4}_{18}$  & 
\multicolumn{2}{l}{$E_1^t=-\frac{t(2+t)}{1+t}e_1-\frac{t(2+t)}{(1+t)^2}e_2$} & \multicolumn{2}{l|}{$E_3^t=-\frac{t^3(2+t)^3}{(1+t)^3}e_3-\frac{2t^3(2+t)^3}{(1+t)^5}e_4$} \\
&&& \multicolumn{3}{l}{$E_2^t=\frac{t^2(2+t)^2}{(1+t)^2}e_2-\frac{t^2(2+t)^2}{(1+t)^3}e_3-\frac{t^2(2+t)^2}{(1+t)^3}e_4$}& \multicolumn{1}{l|}{$ E_4^t=-\frac{t^4(2+t)^4}{(1+t)^5}e_4$} \\\hline
${\bf L}^{4}_{20}$&$\to$&${\bf L}^{4}_{19}$  & 
$E_1^t=t^{-1}e_1$ & $E_2^t=t^{-2}e_2$ &$E_3^t=t^{-3}e_3$& $ E_4^t=t^{-5}e_4$ \\\hline
${\bf L}^{4}_{21}(t^{-1})$&$\to$&${\bf L}^{4}_{20}$  & 
$E_1^t=t^{-1}e_1$ & $E_2^t=t^{-2}e_2$ &$E_3^t=t^{-3}e_3$& $ E_4^t=t^{-5}e_4$ \\\hline

${\bf L}^{4}_{23}(t,\frac{2 \alpha }{2-t})$&$\to$&${\bf L}^{4}_{22}(\alpha)$  & 
\multicolumn{2}{l}{$E_1^t=e_1-\frac{2-t}{t \left(6-t+2 t \alpha -t^2 (1+2 \alpha )\right)}e_2$} & \multicolumn{2}{l|}{$E_3^t=e_3-\frac{4 t \alpha +4 (1+\alpha )-t^2 (1+2 \alpha )}{t \left(6-t+2 t \alpha -t^2 (1+2 \alpha )\right)}e_4$} \\
&&& \multicolumn{3}{l}{$E_2^t=e_2-\frac{(2-t) (1+t)}{t \left(6-t+2 t \alpha -t^2 (1+2 \alpha )\right)}e_3-\frac{(2-t) (2+t (2 \alpha-1 ))}{t^2 \left(6-t+2 t \alpha -t^2 (1+2 \alpha )\right)^2}e_4$}& \multicolumn{1}{l|}{$ E_4^t=e_4$} \\\hline

\multicolumn{7}{|l|}{${\bf L}^{4}_{23}(t+\lambda,\frac{3+\lambda }{(\lambda-1 ) \lambda }) \to {\bf L}^{4}_{24}(\lambda)$} \\ 
 \multicolumn{7}{|l|} {$E_1^t =e_1+\frac{2 t^3 \left(\lambda ^2-9\right)+2 t^2 \left(9-21 \lambda +3 \lambda ^2+\lambda ^3\right)+t \left(3+13 \lambda -16 \lambda ^2+4 \lambda ^3\right)+2 (\lambda-3 )}{2 t (\lambda-3 ) \left(t^2 (3+\lambda )-t \left(3-9 \lambda-2 \lambda ^2\right)-\lambda  \left(3-6 \lambda-\lambda ^2\right)\right)}e_2$} \\

 \multicolumn{7}{|l|} {$E_2^t=e_2+\frac{2 t^3 \left(\lambda ^2-9\right)+t^2 \left(21-29 \lambda -10 \lambda ^2+6 \lambda ^3\right)-t \left(3-18 \lambda +3 \lambda ^2+12 \lambda ^3-4 \lambda ^4\right)-2 \left(3+2 \lambda -\lambda ^2\right)}{2 t (\lambda-3 ) \left(t^2 (3+\lambda )-t \left(3-9 \lambda -2 \lambda ^2\right)-\lambda  \left(3-6 \lambda -\lambda ^2\right)\right)}e_3$}\\
 \multicolumn{7}{|l|}{$ \quad +
\left[\frac{8 (\lambda-3 )^2 \lambda  (1+\lambda )-2 t^7 (3+\lambda )^3 \left(21-13 \lambda +2 \lambda ^2\right)+2 t^6 (3+\lambda )^2 \left(153-387 \lambda+115 \lambda ^2+23 \lambda ^3-8 \lambda ^4\right)+4 t \left(27-27 \lambda -93 \lambda ^2+51 \lambda ^3+66 \lambda ^4-48 \lambda ^5+8 \lambda ^6\right)}{4 t^2 \lambda(\lambda-1 )(\lambda-3 )^2 \left(t^2 (3+\lambda )-t \left(3-9 \lambda-2 \lambda ^2\right)-\lambda  \left(3-6 \lambda-\lambda ^2\right)\right)^2}\right.$} \\
\multicolumn{7}{|l|}{$ \quad- \frac{2 t^5 \left(1134-4806 \lambda +3789 \lambda ^2+2280 \lambda ^3-894 \lambda ^4-284 \lambda ^5+51 \lambda ^6+10 \lambda ^7\right)+t^3 \left(27+999 \lambda -3591 \lambda ^2+2941 \lambda ^3+368 \lambda ^4-1372 \lambda ^5+432 \lambda ^6+20 \lambda ^7-16 \lambda ^8\right)}{4 t^2 \lambda(\lambda-1 )(\lambda-3 )^2 \left(t^2 (3+\lambda )-t \left(3-9 \lambda-2 \lambda ^2\right)-\lambda  \left(3-6 \lambda-\lambda ^2\right)\right)^2}$} \\
\multicolumn{7}{|l|}{$\quad \left.+
\frac{2 t^4 \left(324-2700 \lambda +5085 \lambda ^2-1647 \lambda ^3-2218 \lambda ^4+802 \lambda ^5+141 \lambda ^6-39 \lambda ^7-4 \lambda ^8\right)+2 t^2 \left(54+117 \lambda -201 \lambda ^2+63 \lambda ^3+231 \lambda ^4+28 \lambda ^5-228 \lambda ^6+112 \lambda ^7-16 \lambda ^8\right)}{4 t^2 \lambda(\lambda-1 )(\lambda-3 )^2 \left(t^2 (3+\lambda )-t \left(3-9 \lambda-2 \lambda ^2\right)-\lambda  \left(3-6 \lambda-\lambda ^2\right)\right)^2}\right]e_4$} \\
\multicolumn{6}{|l}{$E_3^t=e_3-\frac{t^3 (3+\lambda )^2 (4 \lambda-13 )+t^2 \left(144-189 \lambda -94 \lambda ^2+19 \lambda ^3+8 \lambda ^4\right)-2 t \left(9-69 \lambda+11 \lambda ^2+37 \lambda ^3-12 \lambda ^4\right)-12 \left(3+2 \lambda -\lambda ^2\right)}{2 t \lambda  (t+\lambda ) \left(3-4 \lambda +\lambda ^2\right) \left(3-6 \lambda-\lambda ^2-t (3+\lambda )\right)}e_4$} &
\multicolumn{1}{l|}{$E_4^t=\frac{2 (\lambda-3 )+t (3+\lambda )}{2 \lambda  \left(3-4 \lambda +\lambda ^2\right)}e_4$} \\
 \hline

\end{longtable}}

 \end{Proof}

\end{document}